\documentclass[11pt,a4paper]{article}

\PassOptionsToPackage{hyphens}{url}\usepackage{hyperref}
\usepackage[url=false,backend=biber, giveninits=true]{biblatex}
\AtEveryBibitem{\clearfield{issn}}

\usepackage{authblk}
\usepackage{amsmath,amssymb,amsfonts,amsthm}
\usepackage{xspace}
\usepackage{lineno} 
\usepackage{dsfont}
\usepackage{mathtools}
\usepackage{multicol}
\usepackage{multirow}
\usepackage{physics}
\usepackage{graphicx}
\usepackage{subcaption}
\usepackage{listings}
\usepackage{relsize}
\usepackage{bigints}
\usepackage[export]{adjustbox}
\usepackage{xcolor}
\usepackage{comment}

\numberwithin{equation}{section}
\theoremstyle{plain}
\newtheorem{theorem}{Theorem}[section]
\newtheorem{lemma}[theorem]{Lemma}
\newtheorem{proposition}[theorem]{Proposition}
\newtheorem{corollary}[theorem]{Corollary}
\theoremstyle{definition}
\newtheorem{definition}[theorem]{Definition}
\newtheorem{assumption}[theorem]{Assumption}
\theoremstyle{remark}
\newtheorem{remark}[theorem]{Remark}

\newcommand{\Ito}{It\^o\xspace}

\newcommand{\Yt}{{\widetilde{Y}}}
\newcommand{\Xc}{x^\ast}
\newcommand{\Zmat}{\mathbb{A}}

\newcommand{\RG}{\mathrm{R_G}} 
\renewcommand{\P}{\mathrm{P}}
\newcommand{\D}{\mathrm{D}}

\newcommand{\BETA}{\boldsymbol{\beta}}
\newcommand{\MU}{\boldsymbol{\mu}}

\newcommand\numberthis{\addtocounter{equation}{1}\tag{\theequation}}

\newcommand{\expect}[1]{\mathbb{E}\left[#1\right]}
\newcommand{\prob}[1]{\mathbb{P}\left[#1\right]}

\author[1]{C. Kelly\thanks{conall.kelly@ucc.ie}}
\author[2]{G. J. Lord\thanks{gabriel.lord@ru.nl}}
\author[3]{M. Ptashnyk\thanks{m.ptashnyk@hw.ac.uk}}
\author[2]{S. Sonner\thanks{stefanie.sonner@ru.nl}}
\affil[1]{School of Mathematical Sciences, University College Cork, Cork, Ireland.}
\affil[2]{Department of Mathematics, IMAPP, Radboud University, Nijmegen, The Netherlands.}
\affil[3]{School of Mathematical and Computer Sciences, Maxwell Institute for Mathematical Sciences, Heriot-Watt University, Edinburgh, UK}

\begin{document}

\title{Mean-square Stability and Bifurcations for Dissipative SDEs
}

\maketitle

\abstract{
We investigate the dynamics of dissipative systems with stochastic forcing and focus in particular on mean-square stability. First we show, under a natural condition on the drift and diffusion, that the stochastic system is mean-square dissipative.  Next we examine the linearised system and state  conditions ensuring that perturbations of a linear system with affine noise are bounded.
We then relate the mean-square dynamics of the nonlinear and linearised systems.
The approach gives a straightforward deterministic method to examine the effects of stochastic forcing on the stability of equilibria of deterministic systems  and to obtain bifurcation diagrams that can be included into standard numerical continuation packages.  The technique is illustrated numerically on some standard and non-standard examples.
}

\medskip
\noindent
\textbf{Keywords:} stochastic differential equation, mean-square stability, stochastic bifurcation, mean-square dissipative.

\noindent
\textbf{MSC codes:} 60H10, 37H20,  60H35, 37H30.

\section{Introduction.}\label{sec:intro}
In the deterministic setting the analysis of mathematical models as dynamical systems and the study of their bifurcations have been highly successful.
There is a growing interest in including the effects of stochastic forcing in models, describing their dynamics and  understanding the effects of the noise, see for example \cite{doi:10.1137/24M1661534,berglund2006noise,BUDD2026135055,carrillo2022noise,Cresson_2016, Cresson_2018,duan2015introduction,FreidlinWentzell,gardiner2010stochastic,vankampen-stochastic-1981,LuxGottschalketal2022,LaingLord} for an overview of models and their analysis. We present here a new approach that complements the methods in these works. Readers interested in learning more about the general theory of stochastic differential equations may also wish to consult \cite{Khasminskii2012,LordKelly2026,LordPowellShardlow,Mao2007}.

Consider the $d$-dimensional system of nonlinear \Ito stochastic differential equations (SDEs)
\begin{equation}\label{eq:main_d}
\begin{split}
dX(t)&= F(X(t))dt + G(X(t)) dW(t), \qquad t\geq 0;\\
X(0)&=X_0\in\mathbb{R}^d,
\end{split}
\end{equation}
where $W=[W_1,\ldots,W_m]^T$ is an $m$-dimensional Wiener process, 
$F:\mathbb{R}^d\to\mathbb{R}^d$ the drift, and $G:\mathbb{R}^d\to\mathbb{R}^{d\times m}$ the diffusion. 
We suppose that the coefficients $F$ and $G$ are sufficiently regular that there exists a unique strong solution of \eqref{eq:main_d} (see Section \ref{sec:lineariz}). 
Although we treat~\eqref{eq:main_d} as an \Ito-type SDE, results can readily be applied to the Stratonovich case by taking into account the addition of an \Ito--Stratonovich correction term, e.g.~\cite{Khasminskii2012,Mao2007}. 
We also suppose that the deterministic equation, i.e.\ when $G(x)\equiv 0$ in  \eqref{eq:main_d}, admits an equilibrium $\Xc\in\mathbb{R}^d$, such that if $X_0=\Xc$ then $X(t)\equiv \Xc$. As a consequence, we have $F(\Xc)=0$, but in general it will not be the case that $G(\Xc)=0$.
Our aim is to investigate the  mean-square stability of the linearisation of \eqref{eq:main_d}, centred around $\Xc$, and to relate it to the (potentially complicated) mean-square dynamics of the nonlinear SDE \eqref{eq:main_d}. 
To this end, we introduce the notion of mean-square dissipativity and impose a condition on the coefficients $F$ and $G$ from which we can prove that \eqref{eq:main_d} has bounded moments for all $t>0$, see for example related works in~\cite{ Khasminskii2012,MATTINGLY2002185,Schurz2000} discussed in Section~\ref{sec:dissipative}.
The analysis is of most interest in the case where the equilibrium $\Xc$ is stable for the deterministic setting.

In the stochastic setting there are various well-established notions of stability for an equilibrium point $\Xc$ of \eqref{eq:main_d}, i.e.~$F(x^*)=G(x^*)=0$. These are normally defined for an equilibrium solution at zero and extend to the nonzero case by applying them directly to the SDE governing the centred process $Y(t)=X(t)-\Xc$, $t\geq 0$, with initial value $Y_0=X_0-\Xc$. We give a short overview of these characterisations of stability in Appendix~\ref{sec:stabilityreview}. Our approach generalizes the notion of mean-square asymptotic stability for equilibria $x^\ast$ of SDEs in a way that is not limited by the requirement that $G(\Xc)=0$, see Definitions~\ref{Def:linear-m-s-stability} and \ref{Def:nonlinear-m-s-stability}. This allows researchers to identify regions of interest in a nonlinear dynamic landscape where more detailed examination is merited, for example for noise induced/delayed bifurcations or for the transition times between equilibria of the deterministic system.

\begin{figure}
    \begin{center}
     \begin{subfigure}[b]{0.48\textwidth}
    \includegraphics[width=\textwidth]{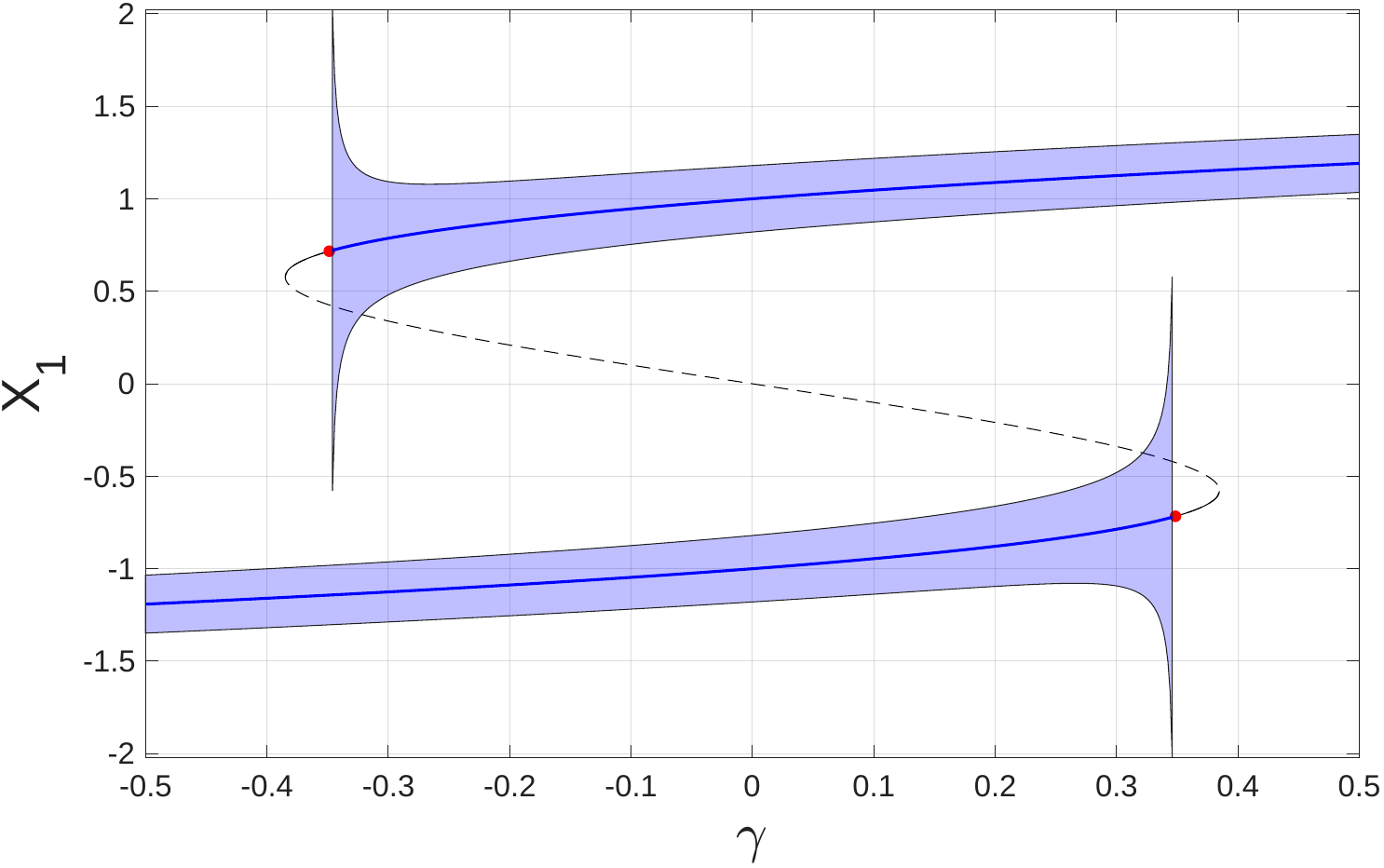}
    \caption{}
    \label{fig:BiStabMultA}
    \end{subfigure}
    \hfill
    \begin{subfigure}[b]{0.48\textwidth}
    \includegraphics[width=\textwidth]{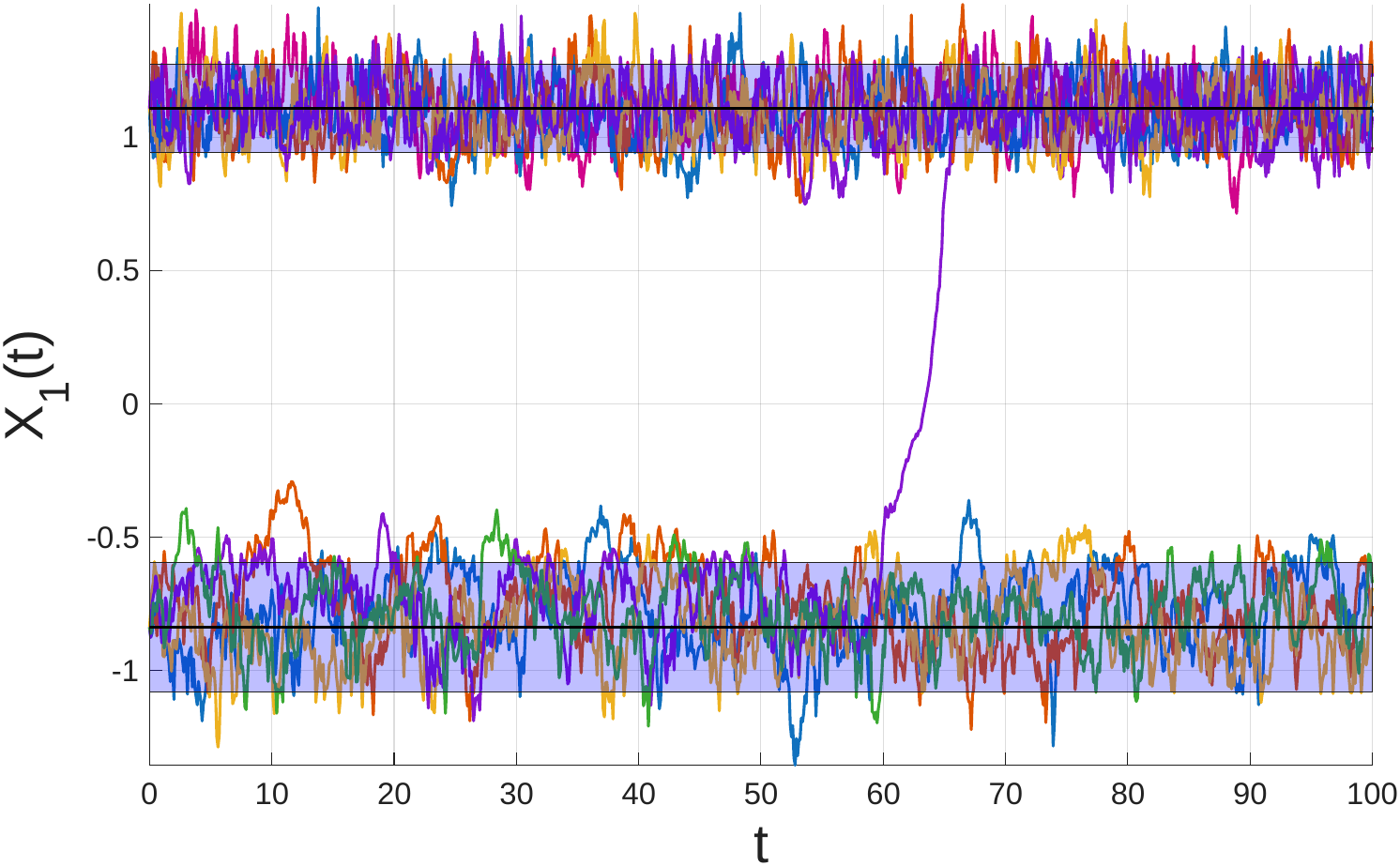}
        \caption{}
    \label{fig:BiStabMultB}
    \end{subfigure}
    \end{center}
    \caption{Bistability SDE example, see Section \ref{sec:BiStable}, with  linear multiplicative noise. In (a) we give a bifurcation diagram as a parameter $\gamma$ is varied for fixed noise intensities. All computations for this diagram are deterministic. In (b) $\gamma=0.25$ is fixed and five samples paths are shown that start close to the underlying equilibria of the deterministic system. }
    \label{fig:BiStabMult}
\end{figure}
As an example consider Figure~\ref{fig:BiStabMultA} which illustrates a bifurcation diagram for an SDE system for which the corresponding deterministic system exhibits bistability. Figure~\ref{fig:BiStabMultB}  shows some sample paths for the SDE. A detailed treatment of this example is given in Section \ref{sec:BiStable}. In Figure \ref{fig:BiStabMultA}, the underlying `S' shaped curve shows the two fold points where the equilibria of the underlying deterministic system have bifurcations (changes in stability). We emphasise that these are not equilibria for the stochastic system. Overlaid on the `S' curve is a representation of the bifurcations in the stochastic system when considered in terms of its mean-square dynamics. When we examine the linearisation of the SDE we observe a change in linear mean-square stability (in our new extended sense) at the two points marked in red, see Corollary~\ref{cor:linear-mean-square-stable}. The shaded area in Figure \ref{fig:BiStabMultA}  then indicates where we can prove the closeness (in mean-square) between the nonlinear SDE and its linearisation, see Corollary~\ref{cor:stabilityX}. The widths of the shaded areas are determined analytically from the respective linearisation in the large-time limit, see Equation~\eqref{eq:BETA}.

In Figure~\ref{fig:BiStabMultB} we illustrate some (numerical) sample paths for $t\in[0,100]$ of the nonlinear SDE. The initial data is taken close to the underlying deterministic equilibria; these are marked by solid horizontal lines. Here we have visual confirmation that the mean-square stability is not to be interpreted in a pathwise sense since, by contrast with the deterministic setting, starting close to the deterministic equilibrium does not mean that an individual trajectory will stay close. Indeed as is well known, trajectories may transition at random times between the two shaded regions as illustrated here, or indeed visit other areas of phase space (see, for example, Figures~\ref{fig:LorenzBif} and \ref{fig:LorenzTraj}  in Section~\ref{sec:lorenz}). However we observe that the variability of the sample paths, while they remain near the deterministic equilibria, closely matches the second moment estimates that we can determine from the linearised systems. The probability for a sample path to remain in a neighbourhood of the deterministic equilibrium can then be estimated using the classical Markov inequality, see Corollary \ref{cor:}.  

There is a large literature on bifurcations and transitions for solutions of SDEs through various lenses, each looking at a single part of the whole picture. For example by considering SDEs as random dynamical systems (see \cite{arnold1974stochastic, EngelKuehn}), by investigating ergodic properties, Lyapunov exponents, and equilibria (see for example \cite{Khasminskii2012,IL1999,IL2001}), and by examining pathwise transitions via large-deviation theory (see for example~\cite{berglund2006noise,FreidlinWentzell,KUEHN20111020}).  Our results complement these approaches. 

The link between stability in probability of the zero equilibrium for a nonlinear SDE with multiplicative noise, and that of its linearisation, is established in~\cite[Theorem 7.1]{Khasminskii2012} and reproduced as Proposition \ref{prop:linnonlin} in Appendix~\ref{sec:stabilityreview}. Our work connects naturally here, as it strengthens the linear-nonlinear connection to include mean-square notions of stability extended to systems with both additive and multiplicative noise. However, pathwise transitions between regions of stability in the state space that are mapped out by our analysis must be investigated by different techniques (see discussion above). 

We stress that our approach to identify regions of mean-square stability is completely deterministic and hence can be readily introduced into standard bifurcation analysis software such as {\tt MatCont}, {\tt pde2path}, {\tt auto-07p}, {\tt xpp}, {\tt PyRates}, {\tt BifurcationKit} etc. Moreover it  does not require the solution of the Fokker-Planck PDE to obtain information on the invariant measure.

The structure of the paper is as follows. In Section \ref{sec:lineariz}, we introduce the mathematical setting and provide a set of conditions for the existence of unique strong solutions of the SDE \eqref{eq:main_d}. In Section \ref{sec:dissipative} we introduce our notion of mean-square dissipativity and formulate an additional condition on the coefficients $F$ and $G$ of \eqref{eq:main_d} that will ensure it. Finally,  we verify that solutions of~\eqref{eq:main_d} then satisfy moment bounds that are uniform in time. 

In Section \ref{sec:centreSDE} we construct the linearised SDE centred around a given equilibrium of the deterministic system, and provide a definition of generalised linear mean-square stability for linear SDEs with affine noise. We also provide a definition of asymptotic mean-square stability for the nonlinear system in terms of its linearisation. 

In Section \ref{sec:MSdynamics} we extend the work of \cite{Buckwar2012} (which in turn built upon \cite{Khasminskii2012}) to construct the ordinary differential equation (ODE) system governing both the first and second moments of linear systems which allows us to determine mean-square stability via an eigenvalue analysis. This is followed in Section \ref{sec:MS_lin_nonlin} by the derivation of an upper bound, exponential in time, on the mean-square linearisation error. This provides sufficient conditions in terms of the derivatives of $F$ and $G$ for when the linearisation may be used as a reasonable guide for the mean-square dynamics of the nonlinear system in a neighbourhood of each deterministic equilibrium point. 

In Section~\ref{sec:numerics} we demonstrate the application of our approach in Sections~\ref{sec:lineariz}--\ref{sec:MS_lin_nonlin} to example systems both theoretically and numerically. Here we present stability and bifurcation diagrams, including noise induced bifurcations. Finally, relevant definitions and results on stochastic dissipativity and stability from the literature are provided in Appendix~\ref{sec:stabilityreview}.

\section{Mathematical setting.}\label{sec:lineariz}

In what follows, for  $x \in \mathbb{R}^d$ and  $\phi\in \mathrm{C}^2(\mathbb{R}^d,\mathbb{R}^d)$, we denote the Jacobian matrix of $\phi$ at $x$ by $\mathbf{D}\phi(x)$, and the second derivative of $\phi$ with respect to $x$ by $\mathbf{D}^2\phi(x)$ which forms a rank-3 tensor. Let $\langle\cdot,\cdot\rangle$ and $\|\cdot\|$ denote the standard Euclidean inner-product and norm in $\mathbb{R}^d$ respectively and $\|\cdot\|_2$ the induced matrix norm. We denote the Frobenious norm of a matrix by $\|\cdot\|_{\mathbf{F}}$ and the norm of the rank-3 tensor over $\mathbb{R}$ by $\|\cdot\|_{T_3}$.
 Recall that if $B$ is a symmetric matrix in $\mathbb{R}^{d\times d}$ then $\lambda_{\text{max}}\{B\}$ denotes the (real valued) largest eigenvalue of $B$, referred to as the principle eigenvalue of $B$.
 
We use common notation from probability that $a\wedge b :=\min\{a,b\}$ and $a \vee b := \max\{a,b\}$. 
Our stochastic processes are in the filtered probability space $(\Omega,\mathcal{F},\mathcal{F}_t,\mathbb{P})$. Here $\Omega$ is the (non-empty) sample space, $\mathcal{F}$ a $\sigma-$algebra of events,
$\mathcal{F}_t$ an increasing family of sub-algebras indexed by time and $\mathbb{P}$ a probability measure.
However, as usual in the context of SDEs  
we suppress the dependence of $X(t,\omega)$ on $\omega\in \Omega$, and simply write $X(t)$. 

Let us first state conditions on  the coefficients of \eqref{eq:main_d} that ensure the existence of unique global solutions, see \eqref{eq:localLipschitz} and \eqref{eq:monotone_fg} below. The polynomial bounds on the second derivatives of $F$ and $G$  in \eqref{eq:D^2f+D^2g} will be used in Section \ref{sec:MS_lin_nonlin} to relate the mean-square dynamics of the nonlinear and linearised systems.
\begin{assumption}\label{assum:FG_bounds}
Suppose that $F$ and $G$ satisfy the following requirements. 
\begin{enumerate}
\item A local Lipschitz condition: for each integer-valued $n\geq 1$ there exists a finite constant $K_n>0$ such that for all $x,y\in\mathbb{R}^d$ with $\|x\|\vee\|y\|\leq n$ 
\begin{equation}\label{eq:localLipschitz}
\|F(x)-F(y)\|^2\vee\|G(x)-G(y)\|^2_{\mathbf{F}}\leq K_n\|x-y\|^2.
\end{equation}
\item A monotone condition: there exists a finite constant $\alpha_1>0$ and $p\geq 2$ such that, for all $x\in\mathbb{R}^d$ 
\begin{equation}\label{eq:monotone_fg}
\langle x,F(x)\rangle+\frac{p-1}{2}\|G(x)\|^2_{\mathbf{F}}\leq \alpha_1(1+\|x\|^2).
\end{equation}
\item There exist constants $c_{1,2}\geq 0, q_{1,2}\geq 0$ such that for all $x\in\mathbb{R}^d$ and $i=1,\dots,m$, we have
\begin{align}
    \big\|\mathbf{D}^2F(x)\big\|_{\mathbf{T}_3}\leq\,\,
    c_1(1+\|x\|^{q_1}), \quad \big\|\mathbf{D}^2G_i(x)\big\|_{\mathbf{T}_3}\leq\,\, c_2(1+\|x\|^{q_2}),\label{eq:D^2f+D^2g}
\end{align}
where $G_i := [G_{1,i}, G_{2,i}, \ldots, G_{d,i}]^T$ denotes  the $i^{\text{th}}$ column of $G$.
\end{enumerate}
\end{assumption}

The first two conditions in Assumption \ref{assum:FG_bounds} are  sufficient to guarantee the existence of a unique strong solution of the SDE \eqref{eq:main_d} for any $T>0$, see for example~\cite{Khasminskii2012}. Moreover, they ensure the following moment bound:

\begin{lemma}\label{lem:L4bound2}
Let $X$ be a solution of \eqref{eq:main_d} and suppose that the conditions \eqref{eq:localLipschitz} and \eqref{eq:monotone_fg} of Assumption \ref{assum:FG_bounds} hold for some $p\geq 2$. Then, if $\alpha_1$ is the constant in \eqref{eq:monotone_fg}, we have
\begin{equation*}
\expect{\|X(t)\|^p}\leq 2^{(p-2)/2}\left(1+\expect{\|X_0\|^p}\right)e^{p\alpha_1 t},\quad t\in[0,T].
\end{equation*}
\end{lemma}
A proof may be found in Chapter 2 of Mao~\cite{Mao2007}, as well as more refined estimates on the moments.

\section{Dissipative stochastic systems.}
\label{sec:dissipative}
Below in our analysis we require bounds on the expectation of the derivatives in \eqref{eq:D^2f+D^2g} along the solutions of the SDE \eqref{eq:main_d} that grow at most polynomially in time. 
This is true, for example, for dissipative systems, and we concentrate on these here. We first introduce a notion of dissipativity for SDEs that applies in the mean-square sense and that extends the standard deterministic notion of having an absorbing ball. Schurz~\cite{Schurz2000} provides one possible such definition that is slightly more general than the notion we use here, for details we refer to Appendix \ref{sec:stabilityreview}.

\begin{definition}
    \label{Def:ms-dissipative}
    Let $X$ be the solution to \eqref{eq:main_d} with initial data $X_0\in\mathbb{R}^d$. Then the SDE~\eqref{eq:main_d} is called mean-square dissipative if there exists $R>0$ such that for all $r>0$ with $\|X_0\|^2<r^2$,  
    there exists a finite time $\mathrm{T}_{r}\geq0$ such that 
    $$\expect{\|X(t)\|^2}<R^2\quad \quad \text{for all}\quad t \geq \mathrm{T}_{r}.$$
\end{definition}
We note that Definition \ref{Def:ms-dissipative} may be adapted to take into account invariant regions of the phase space $\mathbb{R}^d$, see Sections~\ref{sec:fold} and \ref{sec:trans}.

We now formulate conditions on the coefficients of the  SDE \eqref{eq:main_d} that will allow us to prove that it is mean-square dissipative. Again, see Appendix~\ref{sec:stabilityreview} for a brief comparison with the framework of Schurz~\cite{Schurz2000}.
\begin{assumption}
    \label{ass:dissipative}
    Let $p\geq 2$ be given. There exist $\alpha_2>0$ and $\alpha_3\geq 0$ such that
\begin{equation}\label{eq:dissipative}
\langle x, F(x)\rangle + \frac{p-1}{2}\|G(x)\|^2_{\mathbf{F}}  \leq -\alpha_2\|x\|^2+\alpha_3\qquad \text{for all}\quad x\in\mathbb{R}^d.   
\end{equation}
\end{assumption}
Assumption~\ref{ass:dissipative} may be viewed as a dissipative monotonicity condition on the SDE coefficients that extends the notion of a dissipative one-sided Lipschitz bound found, for example, in~\cite{KLOEDEN20121422}. 
We now show that under Assumption~\ref{ass:dissipative} the solution of~\eqref{eq:main_d} has bounded moments for all~$t\geq0$. 
\begin{lemma}\label{lem:dissipativemomentbound}
Suppose Assumption \ref{ass:dissipative} holds and $X$ solves \eqref{eq:main_d}.
Then, for $q\geq  1$, we have 
\begin{equation*}
   \expect{\|X(t)\|^{2q}} \leq 
   \begin{cases}
   \|X_0\|^{2q}e^{-q\alpha_2t}+
   R_q,& \text{if}\ q>1;\\
   \|X_0\|^{2}e^{-2\alpha_2t}+
   R_{1},& \text{if}\ q=1,
   \end{cases}
\end{equation*}
for all $t\geq 0$, where 
\begin{equation*}
   R_q:=\begin{cases}
       \frac{1}{q e}\left(\frac{2\alpha_3}{\alpha_2}\right)^q\left(\frac{q-1}{q}\right)^{q-1},& \text{if}\ q>1;\\
       \frac{1}{e}\frac{\alpha_3}{\alpha_2},& \text{if}\ q=1.
   \end{cases}
\end{equation*} 
\end{lemma}
\begin{proof} 
Consider $\|X(s)\|^{2q}=(X_1(s)^2+X_2(s)^2+\ldots+X_d(s)^2)^q$. We use \Ito's formula (see for example \cite{Mao2007}) to get 
\begin{align}
    \|X(t) & \|^{2q} \nonumber \\
    =&  \|X_0\|^{2q} + 2q \int_0^t\!\! \|X(s)\|^{2(q-1)}\langle X(s), F(X(s))\rangle ds \nonumber \\
   & +  q\int_0^t\!\! \|X(s)\|^{2(q-1)} \|G(X(s))\|_{\mathbf{F}}^2 ds + 
    2q(q-1)\int_0^t\!\! \|X(s)\|^{2(q-2)} \| X(s)^TG(X(s))\|^2  ds \nonumber \\
  &  + 2q \int_0^t\!\! \|X(s)\|^{2(q-1)} \langle X(s),G(X(s)) dW \rangle \nonumber\\
    \leq &   \|X_0\|^{2q} + 2q \int_0^t\!\! \|X(s)\|^{2(q-1)} \Big[ \langle X(s), F(X(s))\rangle + \frac{2q-1}{2} \|G(X(s))\|^2_{\mathbf{F}} \Big] ds \nonumber \\
  &  + 2q \int_0^t \|X(s)\|^{2(q-1)}\langle X(s),G(X(s)) dW \rangle. \label{eq:forlorenz}
\end{align}
Using Assumption \ref{ass:dissipative} it follows that
\begin{align*}
 \|X(t)  \|^{2q} &\leq  \|X_0\|^{2q}+\int_0^t\big[-2q\alpha_2\|X(s) \|^{2q}+2q\alpha_3\|X(s) \|^{2q-2}\big]ds\\
  &\quad  + 2q \int_0^t \|X(s)\|^{2(q-1)}\langle X(s),G(X(s)) dW \rangle,
\end{align*}  
and taking expectation we obtain
\begin{align}\label{eq:estGronwall}
\expect{\|X(t) \|^{2q}}&\leq\|X_0\|^{2q}+\int_0^t\big[-2q\alpha_2\expect{\|X(s) \|^{2q}}+2q\alpha_3\expect{\|X(s) \|^{2q-2}}\big]ds.
\end{align}
If $q=1$ the last term in this inequality is bounded by $2\alpha_3t$ and hence, Gronwall's lemma implies that 
\begin{align*}
    \expect{\|X(t) \|^{2}}&\leq\|X_0\|^{2}e^{-2\alpha_2t}+2\alpha_3te^{-2\alpha_2t}.
\end{align*}
Computing the maximum of the function $2\alpha_3te^{-2\alpha_2t}$ implies the stated bound for $q=1$. 
If $q>1$ we use Young's inequality to estimate the last term in \eqref{eq:estGronwall} and obtain
\begin{align*}
\expect{\|X(t) \|^{2q}}&\leq \|X_0\|^{2q}+\int_0^t(-2q\alpha_2+\varepsilon\tfrac{q-1}{q})\expect{\|X(s) \|^{2q}}ds+C_\varepsilon\frac{(2q\alpha_3)^q} q t,
\end{align*}
for $\varepsilon>0$, where we used Young's inequality in the last step and the constant $C_\varepsilon=\varepsilon^{1-q}$.  Choosing $\varepsilon$ such that $\varepsilon\tfrac{q-1}{q}=q\alpha_2$, which implies that $C_\varepsilon=(\tfrac{q-1}{q^2\alpha_2})^{q-1}$, 
yields
\begin{align*}
\expect{\|X(t) \|^{2q}}&\leq \|X_0\|^{2q}+\int_0^t(-q\alpha_2)\expect{\|X(s) \|^{2q}}ds+\left(\frac{q-1}{q}\right)^{q-1}\left(\frac{2\alpha_3}{\alpha_2}\right)^q\alpha_2  t.
\end{align*}
Hence, Gronwall's lemma implies that 
\begin{align*}
       \expect{\|X(t)\|^{2q}} &\leq \|X_0\|^{2q} e^{-q\alpha_2 t}+ \left(\frac{q-1}{q}\right)^{q-1}\left(\frac{2\alpha_3}{\alpha_2}\right)^q\alpha_2e^{-q\alpha_2t}t\\
       &\leq \|X_0\|^{2q} e^{-q\alpha_2t}
+\left(\frac{q-1}{q}\right)^{q-1}\left(\frac{2\alpha_3}{\alpha_2}\right)^q\frac{1}{q e}.
   \end{align*}
\end{proof}
\begin{corollary}
\label{cor:ms-dissipative}
    The SDE \eqref{eq:main_d} under Assumption \ref{ass:dissipative} is mean-square dissipative in the sense of Definition \ref{Def:ms-dissipative}.
\end{corollary}
By an application of the standard Markov inequality, see for example \cite[Chapter 2]{LordKelly2026}, we obtain the following corollary. 
\begin{corollary}
\label{cor:}
    Under Assumption \ref{ass:dissipative} solutions to \eqref{eq:main_d} satisfy for given $\varepsilon>0$ 
    $$\prob{\|X(t)\|^{2q} > \rho}\leq \varepsilon\qquad \text{for all } t\geq 0,$$
    where $\rho\geq (
    R_q+\|X_0\|^{2q})\varepsilon^{-1}$ and $R_q$ is given in Lemma~\ref{lem:dissipativemomentbound}.  
    Moreover, for given $\varepsilon, \delta>0$ there exists $t^*>0$ such that 
    $$\prob{\|X(t)\|^{2q} > (
    R_q+\delta)\varepsilon^{-1}}\leq \varepsilon,\qquad \text{for all } t\geq t^*.$$
\end{corollary}
As a consequence by taking $\rho$ sufficiently large the probability of the mean-square norm of individual paths getting larger than $\rho$ can be made arbitrarily small, and  the probability that the mean-square norm of individual paths is bounded by $ (R_q+\delta)\varepsilon^{-1}$ is arbitrarily close to one, for large enough times.

\section{Linearisation and mean-square stability.}
\label{sec:centreSDE}

Let $\Xc\in\mathbb{R}^d$ be an equilibrium of the deterministic equation, i.e.~$F(\Xc)=0$, and denote the  solution of \eqref{eq:main_d} centred around $\Xc$ as $Y(t):=X(t)-\Xc$, $t\geq 0$.
The resulting centered SDE  can then be written in integral form as 
\[
Y(t) = X(t) -\Xc= X_0-\Xc + \int_0^t F(X(s)) ds + \int_0^t  G(X(s)) dW(s),\quad t\geq 0.
\]
Assuming appropriate moment bounds on the first and second derivatives of $F$ and $G_i$, the $i$th column of $G$, we may expand each of the drift and diffusion coefficients of \eqref{eq:main_d} as a Taylor series around $\Xc$, 
\begin{equation*}
\begin{aligned}
F(x)& = F(\Xc) + \mathbf{D}F(\Xc) (x- \Xc)+R_F(x-\Xc,\Xc), \\
G_i(x)& = G_i(\Xc) + \mathbf{D}G_i(\Xc)(x-\Xc)+R_{G_i}(x-\Xc,\Xc),
\end{aligned} 
\end{equation*}
for  $i=1,\ldots,m$, where $R_F= (R_{F,j})_{j=1, \ldots d}$ and $R_{G_i}= (R_{G_i, j})_{j=1, \ldots, d}$ are remainder terms. These can be written in integral form as
\begin{equation}
\label{eq:RB}
R_{B}(y,\Xc):= \int_0^1 (1-\tau) y^T\mathbf{D}^2 B(\Xc + \tau y)y d\tau,
\end{equation}
where  $B$ represents either $F$ or $G_j$ for $j =1, \ldots, m$. We introduce the notation  
\begin{equation}\label{eq:RG}
\RG(Y(s),\Xc):=\left[R_{G_1}(Y(s),\Xc),\ldots,R_{G_m}(Y(s),\Xc)\right]\in\mathbb{R}^{d\times m},
\end{equation}
for all $s\geq 0$. Then, using that $F(\Xc)=0$, we can write the centred nonlinear equation 
in integral form as
\begin{multline}\label{eq:main_Y}
Y(t)=Y(0)+\int_{0}^{t}\mathbf{D}F(\Xc)Y(s)ds
\\+\int_{0}^{t}\left(G(\Xc)+\left[\mathbf{D}G_1(\Xc)Y(s),\ldots,\mathbf{D}G_m(\Xc)Y(s)\right]\right)dW(s)
\\+\int_{0}^t R_F(Y(s),\Xc)ds+\int_{0}^{t}\RG(Y(s),\Xc)dW(s),\quad t\geq 0.
\end{multline}
To shorten notation we write it as an SDE,
\begin{equation}\label{eq:main_YSDE}
\begin{split}
dY(t)&=f(Y(t))dt+g(Y(t))dW(t),\quad t\geq 0;\\
Y(0)&=Y_0=X_0-\Xc,
\end{split}
\end{equation}
where we introduced the new coefficients
\begin{equation*}
\begin{split}
f(y)&:=\mathbf{D}F(\Xc)y+R_F(y,\Xc);
\\g(y)&:=G(\Xc)+[\mathbf{D}G_1(\Xc)y,\ldots,\mathbf{D}G_m(\Xc)y]+\RG(y,\Xc).
\end{split}
\end{equation*}

The linearisation of \eqref{eq:main_d} centred around $\Xc$ may then be defined as the process $\Yt$ satisfying the SDE obtained by deleting the remainder terms in \eqref{eq:main_Y},
\begin{multline*}
\Yt(t)=\Yt(0)+\int_{0}^{t}\mathbf{D}F(\Xc)\Yt(s)ds\\
+\int_{0}^{t}\left\{G(\Xc)+\left[\mathbf{D}G_1(\Xc)\Yt(s),\ldots,\mathbf{D}G_m(\Xc)\Yt(s)\right]\right\}dW(s),
\end{multline*}
where $\Yt(0):=Y_0=X_0-\Xc$. It can be re-written as the SDE 
\begin{equation}\label{eq:lin_gen}
    d\Yt(t) =  \mathbf{D}F(\Xc)\Yt(t) dt + \left(G(\Xc)+\left[\mathbf{D}G_1(\Xc)\Yt(t),\ldots,\mathbf{D}G_m(\Xc)\Yt(t)\right]\right) dW(t).
\end{equation}
\begin{remark}
    Note that for \eqref{eq:main_d} with additive noise, i.e.~$G\equiv G_0\in\mathbb{R}^{d\times m}$ for the linearisation centred around $\Xc$ we have 
    $$
    g(y)= G_0,
    $$
    while for an SDE with linear multiplicative noise, i.e.~$G(X)= G_0X$ with $G_0\in\mathbb{R}^{d\times m}$, we obtain affine noise, 
    $$
    g(y)= G_0\Xc+G_0 y.
    $$
    In particular, if $\Xc\neq0$ then $\|\Xc\|$ impacts the intensity of the noise in \eqref{eq:lin_gen}. 
\end{remark}
We now introduce our notions of mean-square stability.  First we consider the stability of the linearisation. 
\begin{definition}\label{Def:linear-m-s-stability}
Let $\Yt$ be the solution of \eqref{eq:lin_gen} with initial data $\Yt(0)\in\mathbb{R}^d$. 
    The point $\Xc\in\mathbb{R}^d$ is said to be linearly mean-square  stable for the system \eqref{eq:main_d} if  there exist a constant $\alpha \geq 0$ and a time $t^\ast>0$ sufficiently large such that for all $\Yt(0)\in\mathbb{R}^d$
$$\expect{\|\Yt(t)\|^2} \leq \alpha^2 \quad \text{for all } t\geq t^\ast.$$
We then  define $\BETA>0$ as    \begin{equation}
    \label{eq:BETA}
    \BETA^2:=\lim_{t\to\infty} \expect{\|\Yt(t)\|^2}<\infty.
    \end{equation}
\end{definition}
Stability in this sense ensures that, for the linear system, perturbations to the point $\Xc$ are controlled, though we note that $\BETA$ cannot be made arbitrarily small by manipulation of the initial value $\Yt(0)$, and this contrasts with classical definitions of stability (see Definition \ref{def:classicalStochStab}). 
Nonetheless, when $\BETA=0$ then the equilibrium is asymptotically mean-square stable in the classical sense (again see Definition~\ref{def:classicalStochStab}). Moreover, via Theorem~\ref{thm:MSlin_gen} and Remark~\ref{Rem:linearizationSDE}, the constant $\BETA$ may be characterized in terms of the additive noise coefficient $G(\Xc)$. Therefore, Definition~\ref{Def:linear-m-s-stability} may be viewed as an extension of the classical notion of asymptotic mean-square stability to affine noise systems. 

The next definition relates the mean-square stability of the SDE \eqref{eq:main_d} and its linearisation~\eqref{eq:lin_gen}.

\begin{definition}
\label{Def:nonlinear-m-s-stability}
The point $\Xc\in \mathbb R^d$ is said to be nonlinearly mean-square stable for the system \eqref{eq:main_d} if it is linearly mean-square  stable and 
    $$\lim_{t\to\infty} \expect{\|Z(t)\|^2}=0,$$
    where $Z$ is the stochastic process defined pathwise by 
    \begin{equation}
    \label{eq:Zdiff}
    Z(t):=Y(t)-\Yt(t),\quad t\geq 0.    
   \end{equation}
\end{definition}
That is, when $\Xc$ is nonlinearly mean-square  stable the nonlinear dynamics can be proved to be close to the linear dynamics as $t\to\infty$ in mean-square. In Section \ref{sec:MSdynamics} we examine the dynamics of the linear system and in Section \ref{sec:MS_lin_nonlin} we relate this to the dynamics of the nonlinear system. 

\section{The mean-square stability of linear SDE systems.}\label{sec:MSdynamics}
We start by gathering results on linear SDEs and moments before applying them in Corollary~\ref{cor:linear-mean-square-stable} to formulate conditions for linear mean-square stability.

Consider a linear SDE  of the form
\begin{align}\label{eq:sdeLinMult}
\begin{split}
\mathrm{d} \Yt(t)&= (A\Yt(t)+\Lambda)\mathrm{d} t + \sum_{i=1}^m (B_i \Yt(t)+\Gamma_i) \mathrm{d} W_i(t),\quad t\geq 0,\\
\Yt(0)&=\Yt_0\in\mathbb{R}^d,
\end{split}
\end{align}
where $A\in\mathbb{R}^{d\times d}$, $\Lambda\in\mathbb{R}^d$, $B_i\in\mathbb{R}^{d\times d}$ and $\Gamma_i\in\mathbb{R}^d$ for $i=1,\ldots,m$, and $W=[W_1,\ldots,W_m]^T$ is an $m$-dimensional standard Brownian motion. 
\begin{remark}\label{Rem:linearizationSDE}
We have made a slight abuse of notation by using $\Yt$ in~\eqref{eq:sdeLinMult} rather than defining a new process. However, we will later apply results in this section to the linearization~\eqref{eq:lin_gen}, where
\begin{equation}\label{eq:linearization}
A:=\mathbf{D}F(\Xc)\in\mathbb{R}^{d\times d};\quad B_i:=\mathbf{D}G_i(\Xc)\in\mathbb{R}^{d\times d};\quad \Lambda=0\in\mathbb{R}^d;\quad \Gamma_i:=G_i(\Xc)\in\mathbb{R}^d.
\end{equation}
\end{remark}
Arnold~\cite{arnold1974stochastic} showed that the mean vector 
and second moment matrix of the linear SDE satisfy the ODE systems given in the next lemma.

\begin{lemma}[Theorem 8.5.5, Arnold~\cite{arnold1974stochastic}]\label{lem:arnold}
    Let $\Yt$ be a solution of the SDE \eqref{eq:sdeLinMult}. Then the following holds:
    \begin{enumerate}
        \item The mean vector $M(t):=\expect{\Yt(t)}$, $t\geq 0$, is the unique solution of the ODE
        \[
            M'(t)=AM(t)+\Lambda,
        \]
        with initial value $M(0)=\Yt(0)$.
        \item The second moment matrix $P(t):=\expect{\Yt(t)\Yt(t)^T}$, $t\geq 0$, is the unique nonnegative definite symmetric solution of the equation
        \begin{multline}\label{eq:Pode}
            P'(t)=AP(t)+P(t)A^T+\Lambda M(t)^T+M(t)\Lambda^T
            \\+\sum_{i=1}^{m}\left(B_i P(t)B_i^T+B_i M(t)\Gamma_i^T+\Gamma_i M(t)^T B_i^T+\Gamma_i\Gamma_i^T\right),
        \end{multline}
        with initial value $P(0)=\Yt(0)\Yt(0)^T$.
    \end{enumerate}
\end{lemma}

We recall the following notation.
\begin{itemize}
\item 
The vectorisation $\mathrm{vec}(A)$ of a matrix $A=(a_{ij})\in\mathbb{R}^{m\times n}$ transforms it into an $mn\times 1$ column vector by stacking the columns of $A$ on top of one another, so that
$$\mathrm{vec}(A) := (a_{11}, a_{12}, \ldots, a_{1n}, a_{21}, \ldots, a_{mn})^T.
$$
\item The Kronecker product of  matrices $A\in\mathbb{R}^{m\times n}$ and  $B\in\mathbb{R}^{p\times q}$ is the
matrix defined by
\begin{equation*}
A \otimes B := \left( \begin{array}{ccc}
 a_{11} B & \dots & a_{1n} B\\
  \vdots & \ddots & \vdots \\
 a_{m1} B & \dots & a_{mn} B
\end{array}\right) \in\mathbb{R}^{mp\times nq}\,.
\end{equation*}
\end{itemize}

The proof of our main result in this section relies upon the following properties of Kronecker products and the vectorisation of a matrix.
\begin{lemma}[Magnus \& Neudecker~\cite{magnus99}]\label{lem:KP_vec}
Let $A$, $B$, $C$, $D$ be real matrices.
\begin{enumerate}
    \item If the matrices $A+B$ and $C+D$ exist, then 
    $$(A+B)\otimes (C+D)=A\otimes C+A\otimes D+B\otimes C+B\otimes D.$$ 
    \item  If $A$, $B$, $C$ are three matrices, such that the matrix product $ABC$ exists, then $$\mathrm{vec}(ABC)=(C^T\otimes A)\mathrm{vec}(B).$$
\end{enumerate}
\end{lemma}
\begin{lemma}\label{lem:vec_KP}
If $A\in\mathbb{R}^{m\times n}$ and $B\in\mathbb{R}^{n\times q}$, then 
\[
\mathrm{vec}(AB)=(B^T\otimes\mathbb{I}_d)\,\mathrm{vec}(A)=(\mathbb{I}_d\otimes A)\mathrm{vec}(B).
\]
Moreover,  if $x\in\mathbb{R}^d$ is a column vector, then $\mathrm{vec}(x)=x$ and $\mathrm{vec}(x^T)=x$.
\end{lemma}
\begin{proof}
The proof is a special case of Lemma \ref{lem:KP_vec}, Part 2.
\end{proof}

The following theorem generalises a result of Buckwar \&  Sickenberger~\cite{Buckwar2012}. We consider the solution $P$ of \eqref{eq:Pode} and note that 
\begin{multline*}
\mathrm{vec}(P)
\\=\left[\expect{\Yt_1^2},\expect{\Yt_2\Yt_1},\ldots,\expect{\Yt_d\Yt_1},\expect{\Yt_1\Yt_2},\expect{\Yt_2^2},\ldots,\expect{\Yt_d\Yt_2},\ldots,\expect{\Yt_d^2}\right]^T.
\end{multline*}
\begin{theorem}\label{thm:MSlin_gen}
Let $\Yt$ be the solution of \eqref{eq:sdeLinMult}, $M$ be as defined in Lemma \ref{lem:arnold} and let $P$ be the solution of \eqref{eq:Pode}. 
Construct a new vector-valued process $Q(t):=[\mathrm{vec}(P(t))^T,M(t)^T]^T$, $t\geq 0$, and set 
\begin{align*}
\tilde{A}&:=(\mathbb{I}_d\otimes A)+(A\otimes\mathbb{I}_d);
\\ \tilde{B}_i&:=(\Gamma_i\otimes B_i)+(B_i\otimes \Gamma_i);
\\ \tilde{C}_i&:=(\Gamma_i\otimes \mathbb{I}_d)\Gamma_i.
\end{align*}
Then $Q$ satisfies the linear ODE
\begin{equation}\label{eq:linearsystem}
Q'(t)=\Zmat Q(t)+S,
\end{equation}
where $\Zmat \in\mathbb{R}^{(d^2+d)\times(d^2+d)}$ is the block matrix  
\begin{equation}\label{eq:Zdefn}
\Zmat:= \begin{pmatrix}
    \tilde{A} +\sum_{i=1}^m B_i\otimes B_i& (\mathbb{I}_d\otimes \Lambda)+(\Lambda\otimes\mathbb{I}_d)+\sum_{i=1}^m\tilde{B}_i\\
    0 & A
\end{pmatrix},
\end{equation}
and $S\in\mathbb{R}^{d^2+d}$ is the stacked vector $\left[\left(\sum_{i=1}^m\tilde{C_i}\right)^T,\,\Lambda^T\right]^T$. 
\end{theorem}
\begin{remark}
Theorem \ref{thm:MSlin_gen} generalises naturally to the case where the coefficients $A,\Lambda,B_i,\Gamma_i$, for $i=1,\ldots,m$, depend on time. For notational simplicity we present the constant coefficient case here.
\end{remark}
\begin{proof}
Vectorise the matrices on both sides of \eqref{eq:Pode} term by term:
\[
\mathrm{vec}(AP(t))=(\mathbb{I}_d\otimes A)\,\mathrm{vec}(P(t));\quad \mathrm{vec}(P(t)A^T)=(A\otimes\mathbb{I}_d)\,\mathrm{vec}(P(t)),
\]
so that
\[
\mathrm{vec}(AP(t)+P(t)A^T)=((\mathbb{I}_d\otimes A)+(A\otimes\mathbb{I}_d))\,\mathrm{vec}(P(t)).
\]
Next,
\[
\mathrm{vec}\left(\sum_{i=1}^m B_i P(t) B_i^T\right)= \sum_{i=1}^{m}\mathrm{vec}(B_iP(t)B_i^T)=\sum_{i=1}^{m}(B_i\otimes B_i)\,\mathrm{vec}(P(t)).
\]
Then,
\begin{align*}
\mathrm{vec}\left(\Lambda M(t)^T+M(t)\Lambda^T\right)&=\left((\mathbb{I}_d\otimes \Lambda)+(\Lambda\otimes\mathbb{I}_d)\right)M(t);
\\ \mathrm{vec}\left(\sum_{i=1}^m B_i M(t)\Gamma_i^T\right)&=\sum_{i=1}^{m}\mathrm{vec}(B_iM(t)\Gamma_i^T)=\sum_{i=1}^{m}(\Gamma_i\otimes B_i)\,M(t),
\end{align*}
and
\[
\mathrm{vec}\left(\sum_{i=1}^m\Gamma_i M(t)^T B_i^T\right)=\sum_{i=1}^{m}\mathrm{vec}(\Gamma_i M(t)^TB_i^T)=\sum_{i=1}^{m}(B_i\otimes \Gamma_i)\,M(t).
\]
Finally, 
\[
\mathrm{vec}(\Gamma_i\Gamma_i^T)=(\Gamma_i\otimes \mathbb{I}_d)\Gamma_i.
\]
So we can write
\begin{multline*}
\mathrm{vec}(P'(t))=\tilde{A}\,\mathrm{vec}(P(t)) + \Big(\sum_{i=1}^{m} B_i\otimes B_i \Big) \mathrm{vec}(P(t)) \\ +\Big((\mathbb{I}_d\otimes \Lambda)+(\Lambda\otimes\mathbb{I}_d)+\sum_{i=1}^{m}\tilde{B}_i\Big)M(t)+\sum_{i=1}^{m}\tilde{C}_i.
\end{multline*}
The statement of the theorem follows from this representation and Part 1 of Lemma \ref{lem:arnold}.
\end{proof}
The eigenvalues of the matrix $\Zmat$ determine the linear mean-square stability of the equilibrium $\Xc$. Hence, applying the theorem to the linearisation of the centred SDE \eqref{eq:lin_gen} we obtain the following corollary.
\begin{corollary}\label{cor:linear-mean-square-stable}
Consider \eqref{eq:sdeLinMult} with the coefficients in \eqref{eq:linearization} and let the assumptions of Theorem \ref{thm:MSlin_gen} hold. Let $\Zmat$ be the matrix defined in \eqref{eq:Zdefn} and $\lambda_i$, $i=1,\ldots, d(d+1)$, be the eigenvalues of~$\Zmat$.
Then the equilibrium $\Xc$ is linearly mean-square stable in the sense of Definition~\ref{Def:linear-m-s-stability} if 
$$\mathrm{Re}(\lambda_i)<0 \qquad \text{for all} \quad i=1,\ldots,d(d+1).$$
\end{corollary}

\begin{remark}\label{rem:Qinfty}
The linear ODE \eqref{eq:linearsystem} has the solution 
$$Q(t) = e^{t\Zmat} Q(0) - \Zmat^{-1}(I-e^{t\Zmat})S.$$
We define $Q_\infty=\lim_{t\to\infty} Q(t) = -\Zmat^{-1}S.$
From this limit  $Q_\infty$ we can extract 
$$
\BETA^2=\lim_{t\to \infty}\expect{\|\Yt(t)\|^2}\leq \|\Zmat^{-1}S\|^2,
$$
where $\BETA$ is as given in Definition~\ref{Def:linear-m-s-stability}.     
We also note that for our numerical experiments in Section \ref{sec:numerics} we remove redundant equations from the linear system $\Zmat Q_\infty = -S$ in order to reduce the complexity of computations.
\end{remark}

\section{Relating the linear and nonlinear dynamics.}
\label{sec:MS_lin_nonlin}
We now examine under what conditions we can use the linearised stochastic system \eqref{eq:lin_gen} to study the mean-square dynamics of the
nonlinear SDE \eqref{eq:main_d} around a deterministic equilibrium $\Xc$. Before we investigate the mean-square dynamics of the linearisation error $Z$ given in \eqref{eq:Zdiff} we recall that 
a function $k:[0,\infty)\to\mathbb{R}$ is of subexponential growth if
$$
\lim_{t\to\infty}k(t)e^{-\kappa t}=0\qquad \text{for all } \kappa>0.
$$
We consider the particular case where $\mathbf{D}F(\Xc)$ is diagonalizable, i.e.
\begin{equation}\label{eq:diagonalize}
    \mathbf{D}F(\Xc)=\P\D\P^{-1},
\end{equation}
where $\D$ is the diagonal matrix with the eigenvalues of $ \mathbf{D}F(\Xc)$ and the eigenvectors are normalized so that $\|\P\|_2=\|\P^{-1}\|_2=1$.

\begin{theorem}\label{thm:Z2L2bound2}
Let Assumption~\ref{assum:FG_bounds} and \eqref{eq:diagonalize} hold, $Y$ be the solution of the centered SDE~\eqref{eq:main_YSDE} and $Z$ be the linearization error defined in \eqref{eq:Zdiff}. Suppose there exists a non-decreasing
 function $K:[0,\infty)\to [0,\infty)$ such that
\begin{equation}\label{eq:RFRG_bound}
\expect{\|R_F(Y(t),\Xc)\|^2} \bigvee \expect{\|\RG(Y(t),\Xc)\|_{\mathbf{F}}^2} \leq K(t)  \qquad \text{for all } t\geq 0.
\end{equation}
 Then for any $\delta_1,\delta_2>0$ we have
\begin{equation*}\label{eq:Z2L2bound}
 \expect{\|Z(t)\|^2}\leq 
 \left[\left(1+\delta_1^{-1}+\delta_2^{-1}\right) K(t)t +\|Z(0)\|^2\right]
e^{\MU t}\qquad \text{for all } t\geq 0,
\end{equation*}
where
\begin{equation}
\label{eq:MU}
\MU:=2\lambda_{\max}\{\mathbf{D}F(\Xc)\} + \delta_1 + (1+\delta_2)\|\mathbf{D}G(\Xc)\|_{T_3}^2.  
\end{equation}
In particular, if $K$ is of subexponential growth and $\MU<0$ 
then $\lim_{t\to\infty} \expect{\|Z(t)\|^2}=0.$
\end{theorem}

\begin{proof} 
By the definition of $Z$ in \eqref{eq:Zdiff} we have 
\begin{multline*}
\P^{-1}Z(t)=\P^{-1}Z(0) + \int_{0}^{t} \left(\P^{-1}\mathbf{D}F(\Xc)Z(s)+\P^{-1}R_F(Y(s),\Xc) \right)ds
\\+\int_{0}^{t}\left(\P^{-1}\left[\mathbf{D}G_1(\Xc)Z(s),\ldots,\mathbf{D}G_m(\Xc)Z(s)\right]+\P^{-1}\RG(Y(s),\Xc)\right)dW(s),\quad t\geq 0,
\end{multline*}
and so $z(t)=\P^{-1}Z(t)$ satisfies
\begin{multline*}
z(t)=z(0) + \int_{0}^{t} \left( D z(s)+\P^{-1}R_F(Y(s),\Xc) \right) ds
\\+\int_{0}^{t}\left(\P^{-1}\left[\mathbf{D}G_1(\Xc)\P z(s),\ldots,\mathbf{D}G_m(\Xc)\P z(s)\right]+\P^{-1}\RG(Y(s),\Xc)\right)dW(s),\quad t\geq 0.
\end{multline*}
Applying \Ito's formula and taking expectation we get
\begin{multline}\label{eq:post_Ito_Z}
\expect{\|z(t)\|^2}=\|z(0)\|^2\\
+\int_{0}^{t}\underbrace{\expect{2\langle z(s), D z(s)\rangle}}_{=:I}+\underbrace{\expect{2\langle z(s),\P^{-1}R_F(Y(s),\Xc)\rangle }}_{=:II}+\underbrace{\expect{\text{Tr}\left(g(s)^Tg(s)\right)}}_{=:III}ds,
\end{multline}
where 
\[
g(s)=\underbrace{\P^{-1}\left[\mathbf{D}G_1(\Xc)\P z(s),\ldots,\mathbf{D}G_m(\Xc)\P z(s)\right]}_{g_1(s)}+\underbrace{\P^{-1}\RG(Y(s),\Xc)}_{g_2(s)},\quad s\in[0,t].
\]
We estimate each of the terms $I$, $II$, and $III$ in \eqref{eq:post_Ito_Z} in turn. The first is straightforward as $D$ is a diagonal matrix,
\begin{align*}
\expect{2\langle z(s), Dz(s) \rangle} &\leq 2 \expect{\lambda_{\text{max}}\left\{D\right\}\|z(s)\|^2}.
\end{align*} 
For $II$, an application of the Cauchy-Schwarz inequality followed by the variant of Young's inequality $ab\leq a^2/\left(2\delta\right)+\delta b^2/2$, leads to
\begin{align*}
\langle z(s), \P^{-1}R_F(Y(s),\Xc) \rangle 
&\leq \left\|z(s)\right\|\left\|\P^{-1}R_F(Y(s),\Xc)\right\|\\
&\leq \frac{1}{2\delta_1}\left\|\P^{-1}\right\|_2^2\left\|R_F(Y(s),\Xc)\right\|^2+\frac{\delta_1}{2}\left\|z(s)\right\|^2\\
& =  \frac{1}{2\delta_1}\left\|R_F(Y(s),\Xc)\right\|^2+\frac{\delta_1}{2}\left\|z(s)\right\|^2,
\end{align*}
for any $\delta_1>0$. Taking expectations on both sides and using that $K$ is non-decreasing yields 
\[
\expect{2\left\langle z(s), \P^{-1}R_F(Y(s),\Xc)\right\rangle}\leq \delta_1^{-1}K(t) +\delta_1 \expect{\left\|z(s)\right\|^2},\quad s\in[0,t].
\]
Finally for $III$, we can use  $g(s) = g_1(s)+g_2(s)$ to estimate
\begin{align*}\label{eq:TrgTg}
\text{Tr}(g^T(s)g(s))& =  \left\|g(s)\right\|^2_{\mathbf{F}} = \left\|g_1(s) +g_2(s)\right\|^2_{\mathbf{F}}
 \leq \left\|g_1(s)\right\|^2_{\mathbf{F}} +\|g_2(s)\|^2_{\mathbf{F}} + 2\|g_1(s)\|_{\mathbf{F}}\|g_2(s)\|_{\mathbf{F}}\\ 
& \leq (1+\delta_2)\|g_1(s)\|^2_{\mathbf{F}} + (1+\delta_2^{-1})\|g_2(s)\|^2_{\mathbf{F}} , \numberthis
\end{align*}
for any $\delta_2>0$. For $g_2$ we obtain, using again that $K$ is non-decreasing,
\begin{align*}\label{eq:g2bound}
\expect{\|g_2(s)\|_{\mathbf{F}}^2} & = \expect{\|\P^{-1}\RG(Y(s),\Xc)\|^2_{\mathbf{F}}}\leq \expect{\|\P^{-1}\|_2^2\|\RG(Y(s),\Xc)\|_{\mathbf{F}}^2 }  \leq K(t),
\end{align*}
and for $g_1(s)=\P^{-1}\left[\mathbf{D}G_1(\Xc)\P z(s),\ldots,\mathbf{D}G_m(\Xc)\P z(s)\right]$ we have
$$
\|g_1(s)\|_{\mathbf{F}}^2 \leq \|\P^{-1}\|_2^2\|\left[\mathbf{D}G_1(\Xc)\P z(s),\ldots,\mathbf{D}G_m(\Xc)\P z(s)\right]\|_{\mathbf{F}}^2.
$$
Let us examine the right-hand side above and in particular the term 
\begin{align*}
& \|\left[\mathbf{D}G_1(\Xc)\P z(s), \ldots,\mathbf{D}G_m(\Xc)\P z(s)\right]\|^2_{\mathbf F}  =\sum_{a=1}^d\sum_{b=1}^{m}\Big(\sum_{j=1}^d [\mathbf{D}G_b(\Xc)]_{a,j}(\P z(s))_j\Big)^2
\\
&\quad \leq  \sum_{a=1}^d\sum_{b=1}^{m}  \sum_{j=1}^d [\mathbf{D}G_b(\Xc)]_{a,j}^2 \sum_{j=1}^d (\P z(s))_j^2 \leq \|\mathbf{D}G(\Xc)\|_{T_3}^2\|\P z(s)\|^2
=  \|\mathbf{D}G(\Xc)\|_{T_3}^2\|z(s)\|^2.
\end{align*}
Taking expectations on both sides of \eqref{eq:TrgTg} and using these estimates leads to
\begin{align*}
\expect{\text{Tr}\left(g(s)^Tg(s)\right)} 
& \leq  (1+\delta_2)\|\mathbf{D}G(\Xc)\|_{T_3}^2\expect{\|z(s)\|^2}+(1+\delta_2^{-1})K(t).
\end{align*}
Returning to~\eqref{eq:post_Ito_Z} yields
\begin{align*}
\expect{\|z(t)\|^2}  \leq & \|z(0)\|^2 + \int_0^t \left(2\lambda_{\max}\{\mathbf{D}F(\Xc)\} + \delta_1 + (1+\delta_2)\|\mathbf{D}G(\Xc)\|_{T_3}^2\right) \expect{\|z(s)\|^2} ds \\
& + K(t) \left(1+\delta_1^{-1}+\delta_2^{-1}\right) t.
\end{align*}
An application of Gronwall's inequality implies the first result. 
Finally, note that if $K$ is of subexponential growth then also the function $tK(t)$ is of subexponential growth.  
\end{proof}
When $\MU<0$ we have exponential convergence of the expected linear dynamics to the non-linear dynamics. We now apply the Markov inequality to get some pathwise information.
\begin{corollary}
\label{cor:nonliear-stability-path}
Let the assumptions of Theorem \ref{thm:Z2L2bound2} hold, suppose that $K$ grows subexponentially and $\boldsymbol{\mu}<0$. Given $\rho>0$, $\varepsilon>0$, $\delta>0$ there exists $t^*>0$ such that
$$\prob{\|Z(t)\|^2 > \rho}<\varepsilon \qquad \text{for all } t>t^*,$$
and 
$$\prob{\|Y(t)\|^2 > \rho}\leq \frac{2\BETA^2+\delta}{\rho} + 2 \frac{\varepsilon}{\rho} \qquad \text{for all } t>t^*.$$
\end{corollary}
\begin{proof}
Applying the Markov inequality we obtain
$$
\mathbb{P}(\|Z(t)\|^2 > \rho) \leq \frac{\expect{\|Z(t)\|^2}}{\rho} = \frac{1}{\rho}\left[K(t)\Big(1+\frac{1}{\delta_1}+\frac{1}{\delta_2}\Big) t +\|Z(0)\|^2\right] e^{\MU t}.
$$  
Since $\MU<0$ and the function $K(t)t$ is of subexponential growth there exists $t^\ast$ such that $\mathbb{P}(\|Z(t)\|^2 > \rho)<\varepsilon$
for $t \geq t^\ast$.

The second statement also follows by an application of Markov's inequality and then the triangle inequality
$$
\begin{aligned}
\prob{\|Y(t)\|^2 > \rho}\leq\frac{\expect{\|Y(t)\|^2}}{\rho}\leq\frac{2\expect{\|Z(t)\|^2}}{\rho}+\frac{2\expect{\|\tilde{Y}(t)\|^2}}{\rho}.
\end{aligned}
$$
The result follows as for the first inequality and using the definition of $\BETA$ in \eqref{eq:BETA}.
\end{proof}

Using the mean-square dissipativity we can bound the moments in \eqref{eq:RFRG_bound} in Theorem~\ref{thm:Z2L2bound2} by constants.
Note that moment bounds for the solution  $X$ of the nonlinear SDE \eqref{eq:main_d} immediately give bounds for the solution $Y=X-\Xc$ of the centered SDE via the triangle inequality. 

\begin{lemma}\label{lem:RFGbound}
Suppose Assumptions~\ref{assum:FG_bounds} and \ref{ass:dissipative} hold and let $Y$ be the solution of \eqref{eq:main_YSDE}.
Then there exists a positive constant $K_{\Xc}>0$, which depends on $\Xc$, such that 
\begin{equation*}
\expect{\left\|R_F(Y(t),\Xc)\right\|^2}\leq K_{\Xc} \left(1+R_{q_1+2}+\|X_0\|^{2(q_1+2)}\right),
\end{equation*}
and
\begin{equation*}
\expect{\left\|\RG(Y(t),\Xc)\right\|_{\mathbf F}^2} \leq K_{\Xc} \left(1+R_{q_2+2}+\|X_0\|^{2(q_2+2)}\right), 
\end{equation*}
where $q_1$ and $q_2$ are the exponents in Assumption~\ref{assum:FG_bounds} and $R_q$ is defined in Lemma \ref{lem:dissipativemomentbound}.
\end{lemma}
\begin{proof}
We can estimate the remainder term $R_F$ as
\begin{align*}
    \|R_F(Y(t),\Xc)\|^2&\leq \int_{0}^{1}\left\|(1-\tau)Y(t)^T\mathbf{D}^2F(\Xc+\tau Y(t))Y(t)\right\|^2d\tau
    \\&\leq \int_{0}^{1}(1-\tau)^2\left\|\mathbf{D}^2F(\Xc+\tau Y(t))\right\|^2_{T_3}\left\|Y(t)\right\|^4d\tau.
\end{align*}
Making use of condition \eqref{eq:D^2f+D^2g} in Assumption \ref{assum:FG_bounds} and using Young's inequality there exist positive constants $K_{\Xc}', K_{\Xc}''>0$ depending on $\Xc$ such that
\begin{align*}
    \|R_F(Y(t),\Xc)\|^2&\leq \int_{0}^{1}(1-\tau)^2c_1^2\left(1+\left\|\Xc+\tau Y(t)\right\|^{q_1}\right)^2\left\|Y(t)\right\|^4d\tau\\
    &\leq \int_{0}^{1}(1-\tau)^2K_{\Xc}'\left(1+\|Y(t)\|^{2q_1}\right)\left\|Y(t)\right\|^4d\tau\\
    &
    \leq K_{\Xc}''\left(1+\|Y(t)\|^{2q_1+4}\right). \numberthis \label{eq:hatRFpoly}
\end{align*}
Similarly we obtain for some positive constant $K_{\Xc}'''>0$ 
\begin{equation}\label{eq:hatRGpoly}
    \|R_G(Y(t),\Xc)\|_{\mathbf F}^2\leq K_{\Xc}'''\left(1+\|Y(t)\|^{2q_2+4}\right).
\end{equation}
Taking expected values on both sides of \eqref{eq:hatRFpoly} and \eqref{eq:hatRGpoly} we obtain
\begin{align*}
\expect{\|R_F(Y(t),\Xc)\|^2}&\leq K^0_{\Xc}\left(1+\expect{\|Y(t)\|^{2q_1+4}}\right),
\\ \expect{\|\RG(Y(t),\Xc)\|_{\mathbf F}^2}&\leq K^0_{\Xc}\left(1+\expect{\|Y(t)\|^{2q_2+4}}\right),
\end{align*}
for some $K^0_{\Xc}>0$.
Finally, we apply the moment bound from Lemma \ref{lem:dissipativemomentbound}, with the estimate 
$$
\expect{\|Y(t)\|^p}\leq 2^{p-1}\expect{\|X(t)\|^p}+2^{p-1}\|\Xc\|^p,
$$
to get, after absorbing the term depending on $\Xc$ into the constant $K_{\Xc}$, 
\begin{align*}
\expect{\left\|R_F(Y(t),\Xc)\right\|^2}
&\leq K_{\Xc} \left(1+R_{q_1+2}+\|X_0\|^{2(q_1+2)} e^{-(q_1+2)\alpha_2 t}\right),\\
\expect{\left\|\RG(Y(t),\Xc)\right\|_{\mathbf F}^2}&\leq K_{\Xc} \left(1+R_{q_2+2}+\|X_0\|^{2(q_2+2)} e^{-(q_2+2)\alpha_2 t}\right),
\end{align*}
where $\alpha_2$ is as given in Assumption \ref{ass:dissipative} and $R_q$ in Lemma \ref{lem:dissipativemomentbound}.
Bounding the exponentials above gives the result.
\end{proof}
Combining Lemma \ref{lem:RFGbound} and Theorem \ref{thm:Z2L2bound2} we obtain the following result that allows us to characterize the mean-square stability. 
\begin{corollary}\label{cor:stabilityX}
Let Assumptions~\ref{assum:FG_bounds}, \ref{ass:dissipative} and \eqref{eq:diagonalize} hold. Then the SDE \eqref{eq:main_d} is non-linearly mean-square stable about $\Xc$ whenever $\MU<0$ in \eqref{eq:MU}. 
\end{corollary}

\section{Examples and numerical illustrations.}\label{sec:numerics}

To demonstrate our general results we analyse some examples of typical dynamical systems. The key elements to determine linear and nonlinear mean-square stability either numerically or analytically are as follows.
\begin{enumerate}
\item Determine the mean-square linear stability from $\lambda_{\max} \{ \mathbb{A} \}$.
\item Solve the linear equation $\mathbb{A} Q_\infty = -S $
for $Q_\infty:=\lim_{t\to\infty} Q(t)$, where $\mathbb{A}$, $S$ and $Q$ are as defined in the linear ODE~\eqref{eq:linearsystem} with \eqref{eq:linearization}, see Remark~\ref{rem:Qinfty}.
\item  
Knowing $Q_\infty$  find $\BETA^2=\lim_{t\to \infty}\expect{\|\Yt(t)\|^2}$ in \eqref{eq:BETA} from the corresponding components of $Q_\infty$. 
\item Use \eqref{eq:MU} to compute $\MU$ and hence relate the dynamics of the nonlinear SDE and its linearization.
\end{enumerate}
We stress that this approach does not require any simulation of a stochastic system as every step is purely deterministic. In practice the key quantities to compute are $\BETA$ and $\MU$.

We provide example code in \texttt{octave}, \texttt{matlab} and \texttt{python} to perform these calculations in \url{https://github.com/Gabriel-Lord/MS-stability-and-Bifurcations}.  We note that there is a redundancy in the system \eqref{eq:linearsystem} as, $\expect{\Yt_i\Yt_j}=\expect{\Yt_j\Yt_i}$. It is a relatively straightforward computation to remove this redundancy and consider a smaller linear system to find $Q_\infty$. This is useful when considering large systems of SDEs (see for example Section~\ref{sec:Allen-Cahn}).

In each example we compute  bifurcation diagrams and illustrate some sample trajectories for the SDEs we consider. 
In the figures below we shade the area (a ball with radius $\BETA$ around the deterministic equilibrium $\Xc$) for the parameter regime,  where the linear dynamics and nonlinear dynamics are close, i.e.~where $\boldsymbol{\mu}<0$ in Corollary \ref{cor:stabilityX}.
We indicate with a red dot the point where the linear system changes stability (when available), i.e.~where $\lambda_{\text{max}}\approx 0$ in Corollary \ref{cor:linear-mean-square-stable}.
For sample trajectories we show five sample paths starting in a ball of radius $10^{-6}$
around the deterministic equilibrium $\Xc$.
In the numerical experiments we fix $\delta_1=\delta_2=10^{-3}$ for $\MU$ in \eqref{eq:MU} throughout.
\begin{remark}
    In our implementation to show sample trajectories we used an explicit Euler-Maruyama method with taming, see for example~\cite{HutzenthalerJentzenKloeden,Sabanis}. This method is strongly convergent for a broad class of SDEs, though for large time-steps and large function values it can distort the observed dynamics. Nonetheless it is straightforward to implement with a fixed step and we do so here to generate a small number of trajectories on a fine mesh ($\Delta t=0.01$) for display. However, since the stepsize must be small, it is not a method we would necessarily recommend for applications requiring the simulation of a large ensemble of trajectories. An alternative approach would be to adapt the time step, see for example \cite{FangGiles,Kelly20181523,KellyLord2022} or to use a splitting type method for the particular SDE. In a number of the examples below the SDE solution is constrained to a domain (that is $X(t)\in D$ for all $t\geq 0$ for some set $D\subset \mathbb{R}^d$) that it may be desirable for the numerical discretisation to preserve. For example, the invariant region is the positive cone $D=\mathbb{R}^+$ for the example on pitchfork bifurcation with linear multiplicative noise in Section \ref{subsec:pitchfork}. There are specialised domain preserving numerical methods, see for example \cite{ERDOGAN202630,liu2025ergodicestimatesonestepnumerical,Milstein19981010}, but it is not true in general that a convergent method will ensure this. In addition there are specialized methods available to address the particular numerical challenges associated with the Cox-Ingersoll-Ross model~\eqref{eq:CIR}, and linearly stable methods for the stochastic Allen-Cahn equation~\eqref{sec:Allen-Cahn}.
\end{remark}
We start by considering some standard one-dimensional examples before examining higher dimensional ones:  a two dimensional system displaying bistability, the three dimensional Lorenz equations and a fifty dimensional system of SDEs arising from a finite difference approximation of an SPDE.

\subsection{Some one dimensional examples.}
We start by considering some standard deterministic bifurcation scenarios in one-dimension: ODEs exhibiting Pitchfork, Fold and Transcritical bifurcations to which we add some stochastic forcing. We then examine the Cox-Ingersoll-Ross model, commonly used in financial applications.

\subsubsection{Pitchfork: SDE with additive and multiplicative noise.}\label{subsec:pitchfork}
We consider the following scalar SDE  with either additive or multiplicative one dimensional noise for which  the corresponding ODE has a pitchfork bifurcation, 
$$dX(t) = F(X(t)) dt + G(X(t)) dW(t),$$
where $F(x) =\gamma x - x^3$ and $G$ is one of the following
$$G(x) = \sigma; \quad G(x)= \sigma x; \quad G(x) = \sigma x^2.$$ 
In the deterministic ODE at $\gamma=0$ we have the pitchfork bifurcation and for $\gamma>0$ the equilibrium $\Xc=0$ is unstable and the equilibria $\Xc_{\pm}= \pm \sqrt{\gamma}$ are asymptotically stable.
First we examine mean-square dissipativity (see  Definition~\ref{Def:ms-dissipative}) for the SDEs by showing that Assumption \ref{ass:dissipative} holds and applying Corollary~\ref{cor:ms-dissipative}. 
\begin{enumerate}
\item Additive noise: we have $G(x) = \sigma$ and obtain
$$
\langle x, F(x) \rangle +\frac{p-1} 2 |G(x)|^2 = 
\gamma x^2-x^4 +(p-1) \frac{\sigma^2}{2}  \leq - \alpha_2 x^2+\alpha_3, 
$$
if $\alpha_3\geq (\tfrac{\alpha_2+\gamma}{2})^2+(p-1) \tfrac{\sigma^2}{2}$.
\item Linear multiplicative noise:  considering $G(x) = \sigma x$ gives
$$
\langle x, F(x)\rangle + \frac {p-1} 2 |G(x)|^2 = 
\Big(\gamma + (p-1)\frac{\sigma^2}{2}\Big) x^2-x^4 \leq -\alpha_2 x^2+\alpha_3,
$$
if $\alpha_3\geq \tfrac{1}{4}(\alpha_2+\gamma+(p-1) \tfrac{\sigma^2}{2})^2$.

\item Quadratic multiplicative noise: taking $G(x) = \sigma x^2$ we find 
$$
\langle x, F(x) \rangle + \frac {p-1}2 |G(x)|^2 = \gamma x^2- \Big(1- (p-1) \frac{\sigma^2}{2}\Big) x^4 \leq  -\alpha_2 x^2+\alpha_3,
$$
if $\alpha_3 \geq \tfrac{(\gamma + \alpha_2)^2}{2(2 - (p-1)\sigma^2)} $ and $\sigma^2 < 2/(p-1)$. 
\end{enumerate}
\textbf{Additive Noise.}
The matrices $\mathbb{A}$ and $S$ for the linearised system~\eqref{eq:linearsystem} for this example are 
$$
\mathbb{A} = \begin{pmatrix}
    DF|_{\Xc_{\pm}}+DF|_{\Xc_\pm} & 0 \\
    0  & DF|_{\Xc_\pm} 
\end{pmatrix}, \qquad S =\begin{pmatrix}
    \sigma^2 \\  0 
\end{pmatrix}.
$$
From \eqref{eq:MU} we have
$\MU = -2\gamma + \delta_1 + 0.$
From this we can determine the nonlinear mean-square stability of $\Xc_{\pm}$. 
In fact for this simple one-dimensional example it is straightforward to compute all the steps in the analysis of Sections~\ref{sec:lineariz}--\ref{sec:MS_lin_nonlin} analytically. 

In Figure~\ref{fig:pitchfork_add} we plot in (a) the bifurcation diagram and in (b) sample paths for $\gamma=0.25$ and $\sigma=0.1$.
The region of nonlinear mean-square stability in Figure~\ref{fig:pitchfork_addA} is determined by $\MU$. In Figure~\ref{fig:pitchfork_addB} we have  $x^\ast_{\pm}=\pm 0.5$ and $\BETA^2=\sigma^2/4\gamma=0.1$ as can be observed in the figure.
We recall our observation from the introduction and Corollary~\ref{cor:nonliear-stability-path} and observe that sample paths do not stay close to the deterministic equlibria $\Xc_{\pm}$ even though they are mean-square stable. 
This is consistent with the known long-time behaviour of the SDE. For example it is well-known that the invariant density $p_\infty$ satisfies
$p_\infty(x) \propto \exp\left(\mathcal{V}(x)\right)$ where $\mathcal{V}(x)=\frac{-x^4}{2\sigma^2}+\frac{\gamma x^2}{\sigma^2}$.
Hence in the long term dynamics there are transitions between the two wells in the potential $\mathcal{V}$.
\begin{figure}[htbp]
    \begin{center}
     \begin{subfigure}[b]{0.49\textwidth}
    \includegraphics[width=\textwidth]{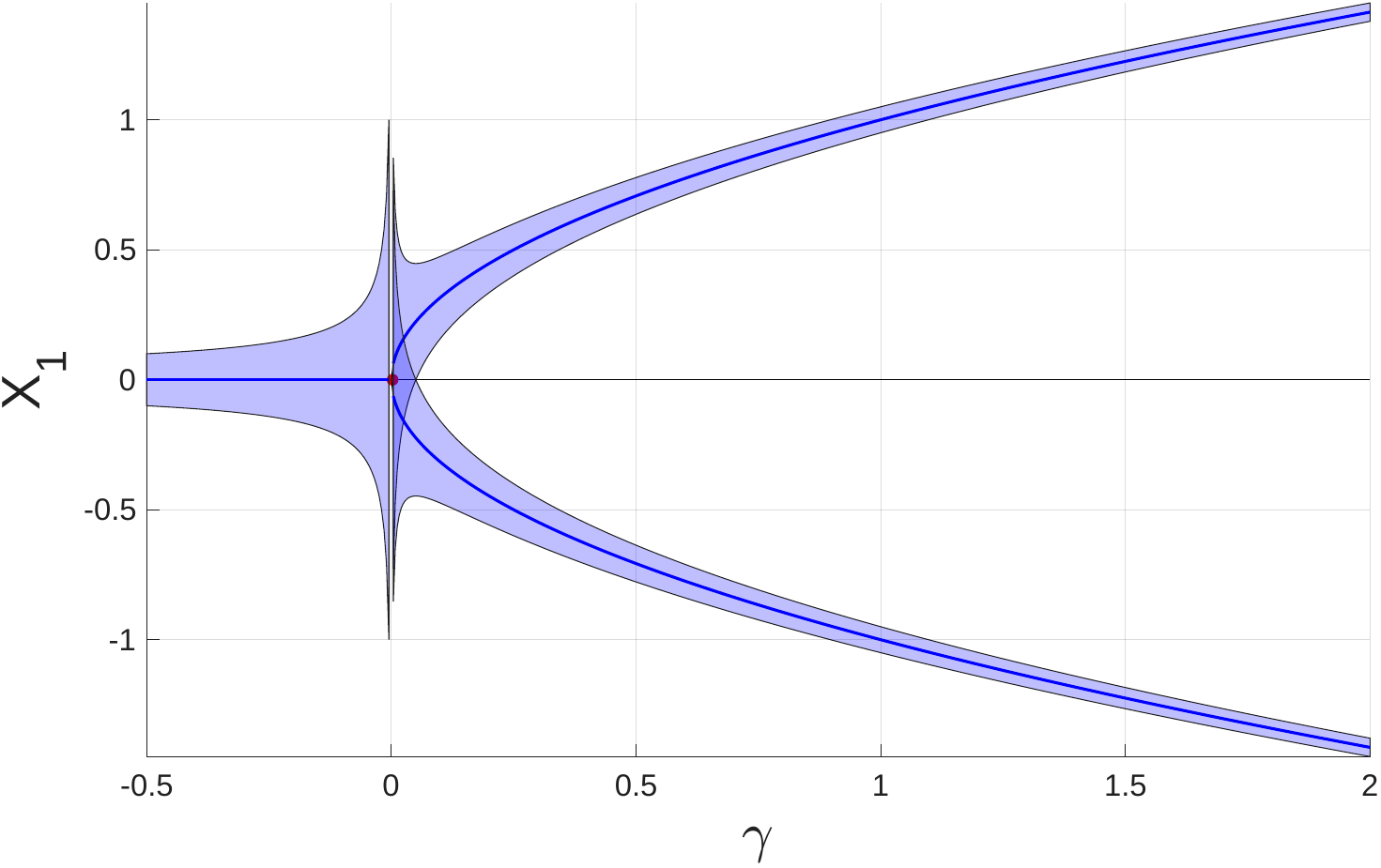}
    \caption{}
    \label{fig:pitchfork_addA}
    \end{subfigure}
    \hfill
    \begin{subfigure}[b]{0.49\textwidth}
     \includegraphics[width=\textwidth]{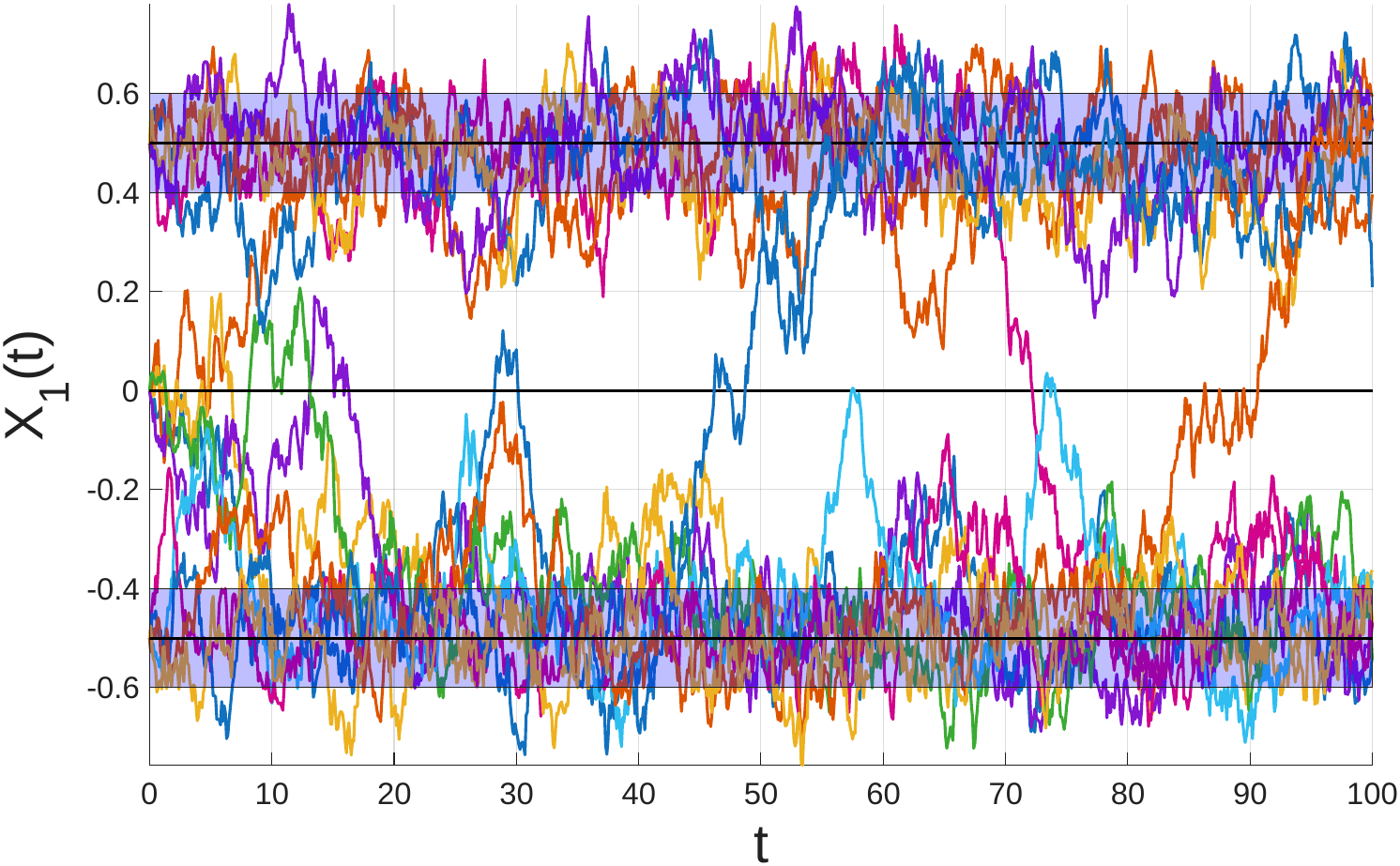}
        \caption{}
    \label{fig:pitchfork_addB}
    \end{subfigure}
    \end{center}
    \caption{Pitchfork SDE with additive noise and $\sigma=0.1$. (a) Bifurcation diagram and (b) sample trajectories at $\gamma=0.25$.}
    \label{fig:pitchfork_add}
\end{figure}

\noindent
\textbf{Linear Multiplicative Noise.}
For linear multiplicative noise we take $G(x)=\sigma x$.

As a numerical illustration we examine the bifurcation diagram for fixed $\sigma$ in Figure~\ref{fig:pitchfork_LinA} as $\gamma$ varies and in Figure~\ref{fig:pitchfork_LinB} as $\sigma$ varies and $\gamma$ is fixed. 
Recall that the (red) circles indicate where the linear system changes stability and the shaded area where the deterministic equilibria $\Xc_\pm$ are non-linearly mean-square stable. We observe here, and in other bifurcation diagrams, that the mean-square stability of the equilibria changes at different values than in the deterministic setting.
It is clearly visible in Figure~\ref{fig:pitchfork_LinB} that as the noise increases the non-linear mean-square stability is lost (in a noise induced bifurcation). 
Sample trajectories are shown for $\gamma=0.25$ and $\sigma=0.1$ in Figure~\ref{fig:pitchfork_LinTrajA} and in
Figure~\ref{fig:pitchfork_LinTrajB} for $\sigma=0.8$. Note that the shaded regions now overlap. Here we see the effect of the increase in $\sigma$ on the variance of the sample trajectories and on $\BETA$. The instability of $\Xc=0$ is clearly visible in Figure~\ref{fig:pitchfork_LinTrajA}.
\begin{figure}[htbp]
    \begin{center}
     \begin{subfigure}[b]{0.49\textwidth}
\includegraphics[width=\textwidth]{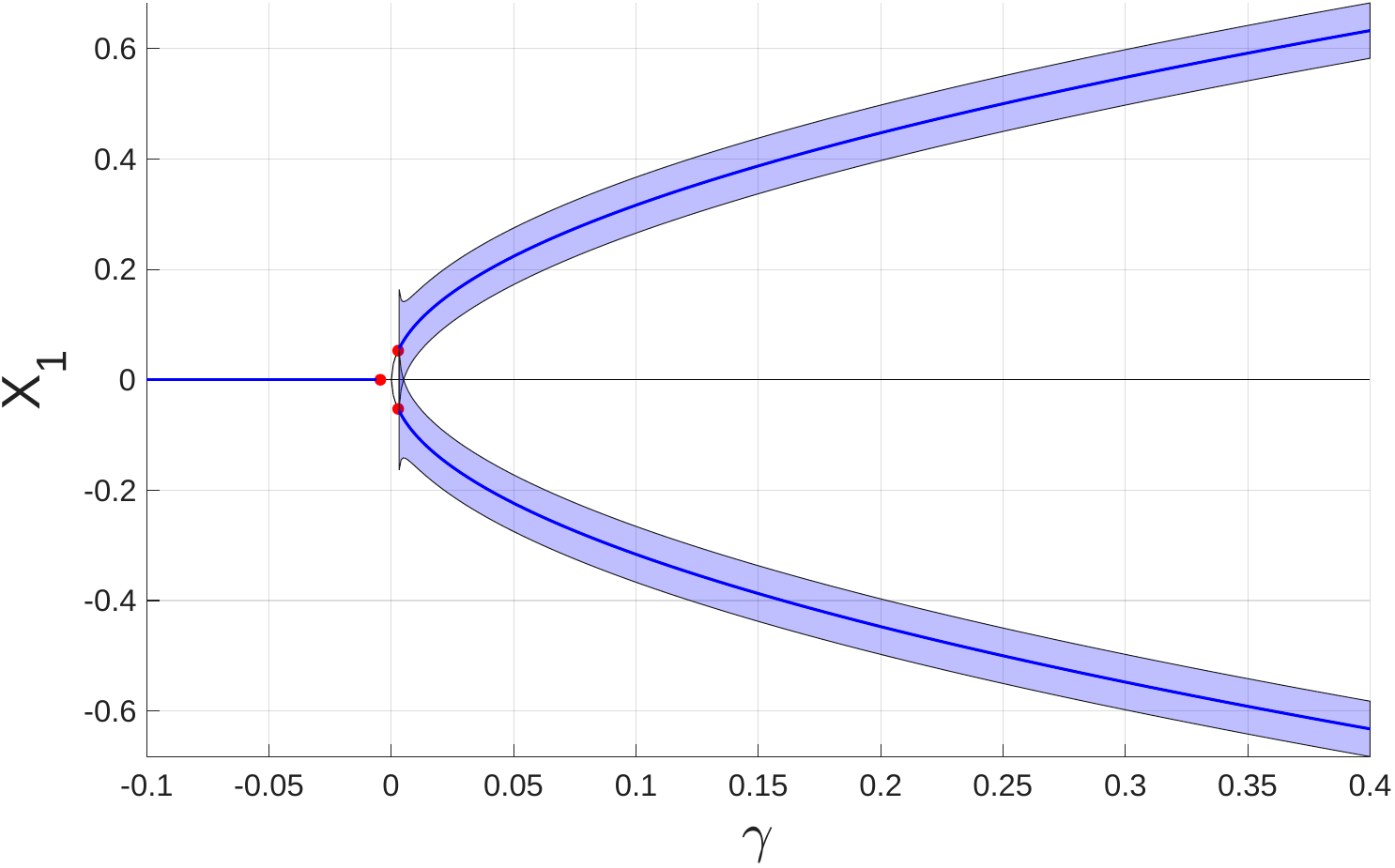}
    \caption{}
    \label{fig:pitchfork_LinA}
    \end{subfigure}
    \hfill
    \begin{subfigure}[b]{0.49\textwidth}
         \includegraphics[width=\textwidth]{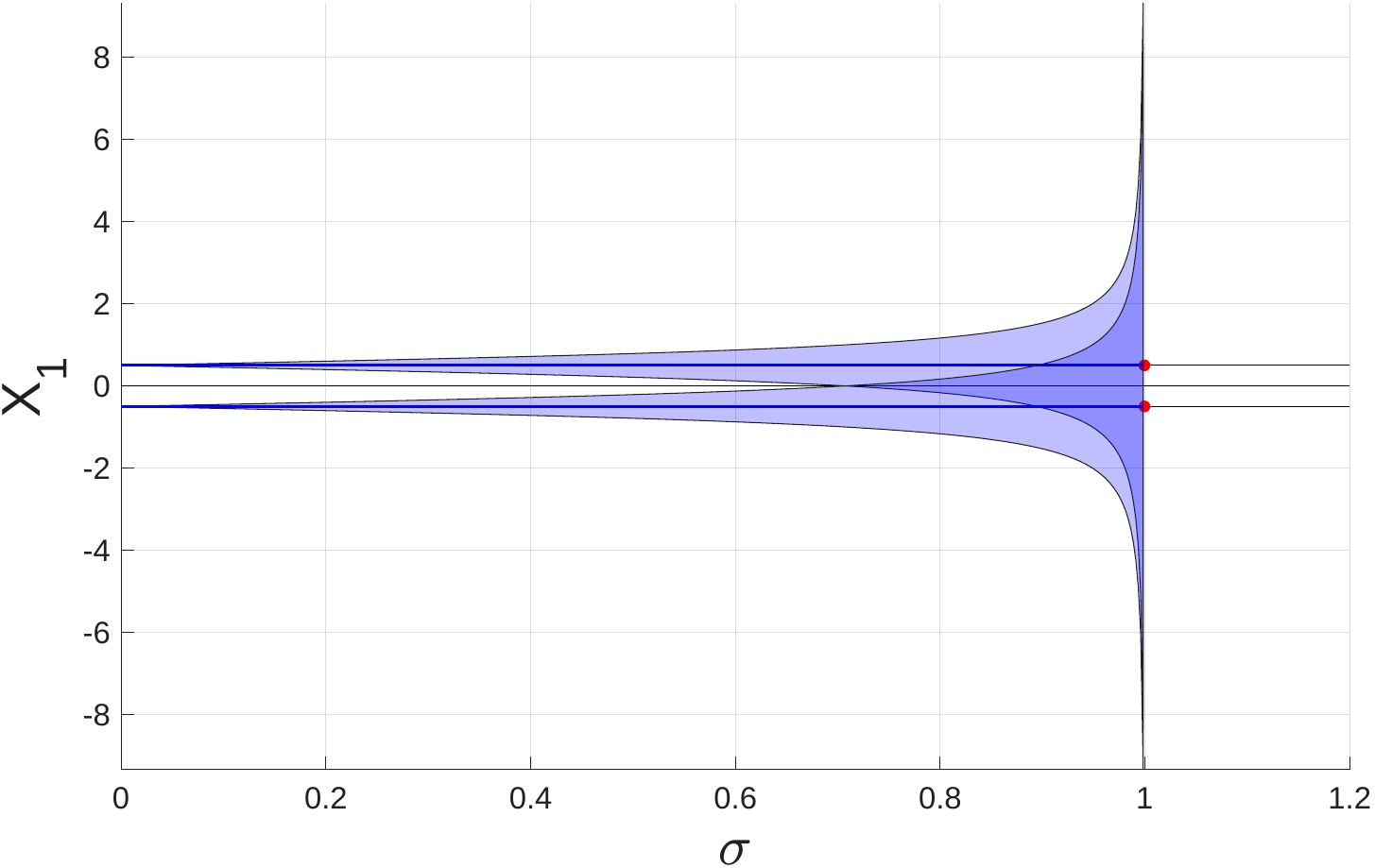}
        \caption{}
    \label{fig:pitchfork_LinB}
    \end{subfigure}
    \end{center}
    \caption{Pitchfork SDE with linear multiplicative noise with (a) $\sigma=0.1$ fixed and varying $\gamma$ and (b) $\gamma=0.25$ fixed and varying $\sigma$.}
    \label{fig:pitchfork_LinBif}
\end{figure}

\begin{figure}[htbp]
    \begin{center}
     \begin{subfigure}[b]{0.49\textwidth}
    \includegraphics[width=\textwidth]{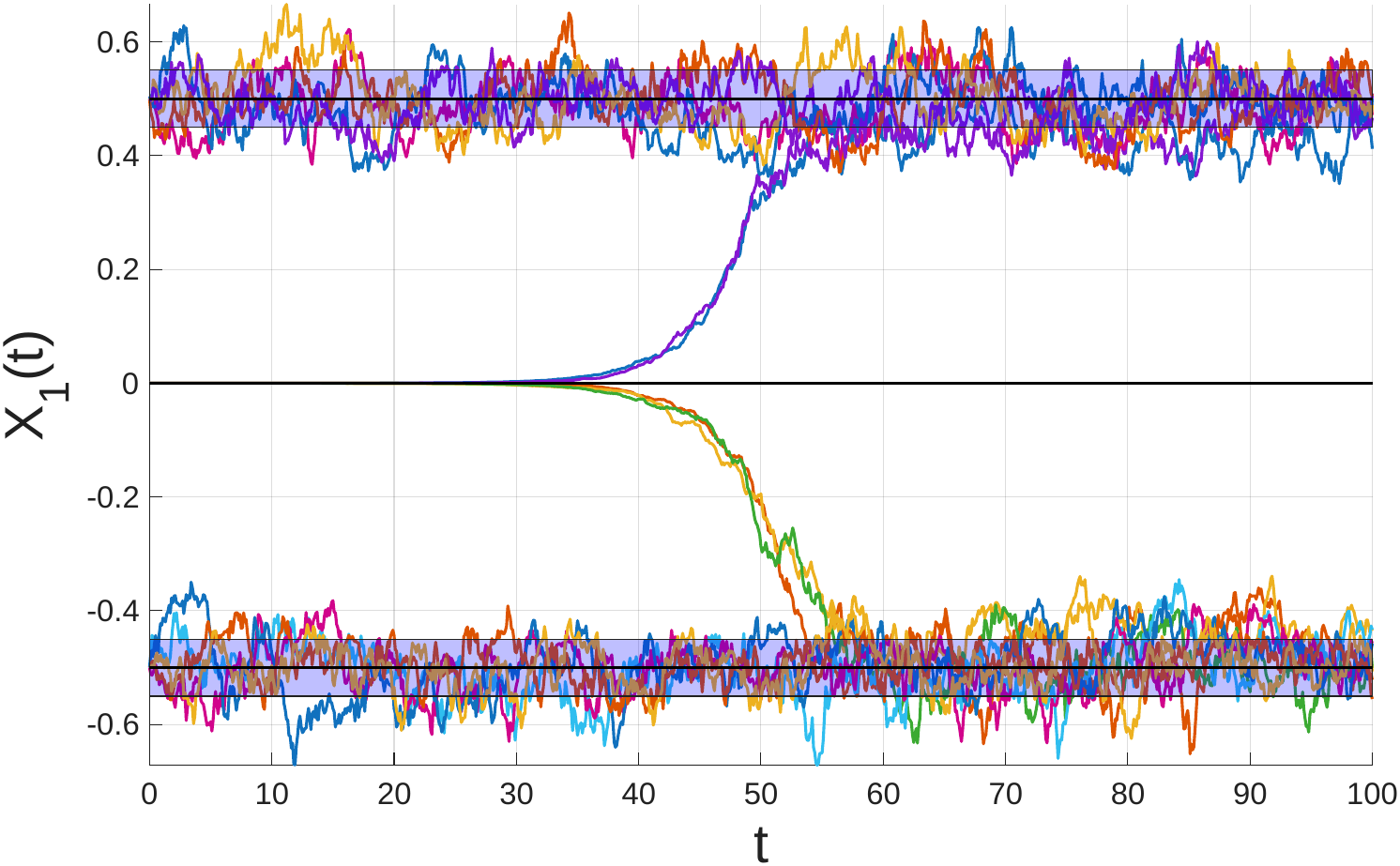}
    \caption{}
    \label{fig:pitchfork_LinTrajA}
    \end{subfigure}
    \hfill
    \begin{subfigure}[b]{0.49\textwidth}

         \includegraphics[width=\textwidth]{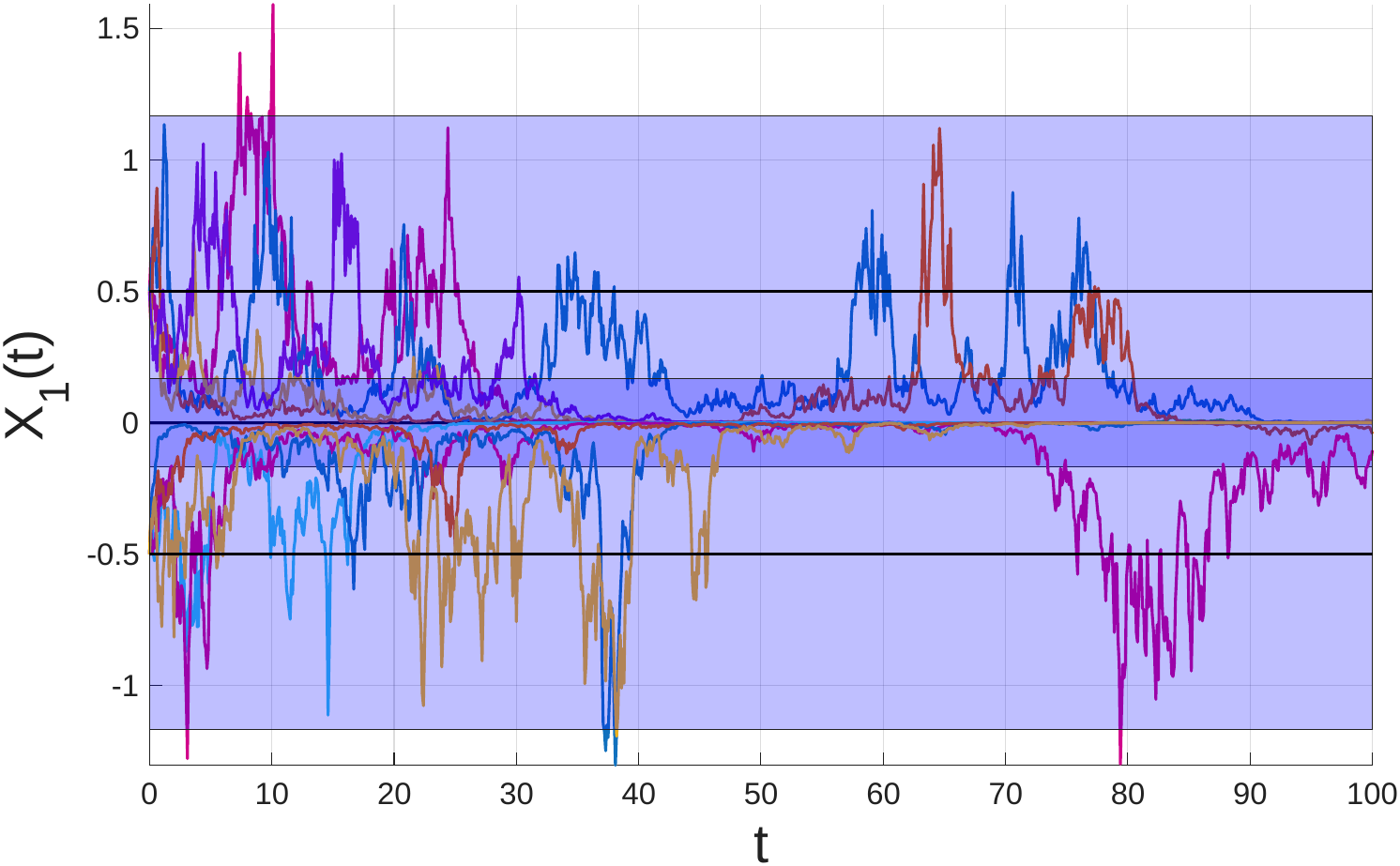}
        \caption{}
    \label{fig:pitchfork_LinTrajB}
    \end{subfigure}
    \end{center}
    \caption{Sample trajectories for linear multiplicative noise with $\gamma=0.25$ and (a) $\sigma=0.1$ fixed and  (b) $\sigma=0.8$ fixed.}
    \label{fig:pitchfork_LinTraj}
\end{figure}

\noindent
\textbf{Quadratic Multiplicative Noise.}
We now take $G(x)=\sigma x^2$. We examine numerically the bifurcation diagram for fixed $\sigma$ as $\gamma$ varies, Figure~\ref{fig:pitchfork_quadA}, and illustrate some sample trajectories in Figure~\ref{fig:pitchfork_quadB}. 
The plots are similar to those for linear multiplicative noise however we observe that $\BETA$ is smaller (compare Figure~\ref{fig:pitchfork_LinTrajA} and Figure~\ref{fig:pitchfork_quadB} for example).
\begin{figure}[htbp]
    \begin{center}
     \begin{subfigure}[b]{0.49\textwidth}
    \includegraphics[width=\textwidth]{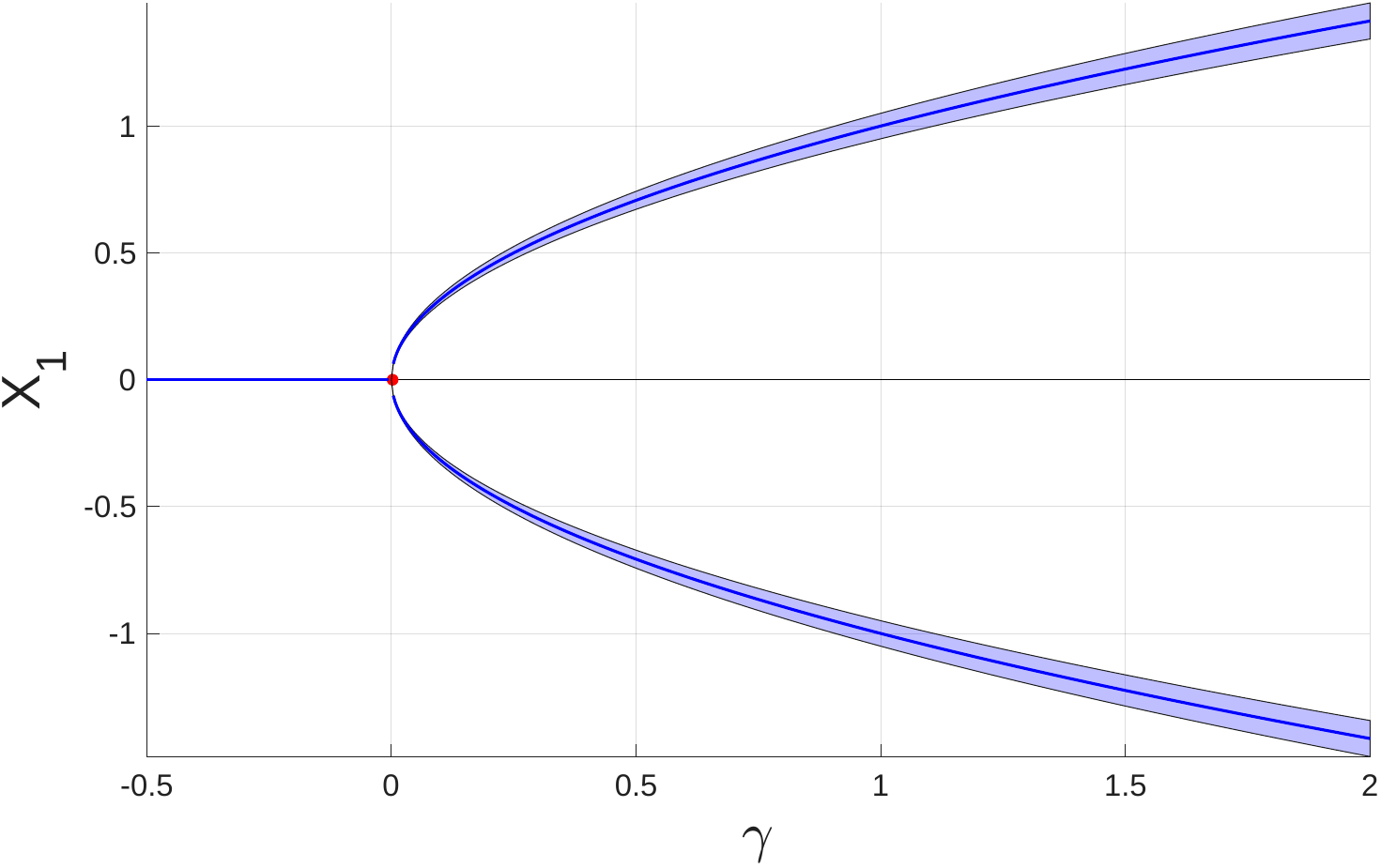}
    \caption{}
    \label{fig:pitchfork_quadA}
    \end{subfigure}
    \hfill
    \begin{subfigure}[b]{0.49\textwidth}
     \includegraphics[width=\textwidth]{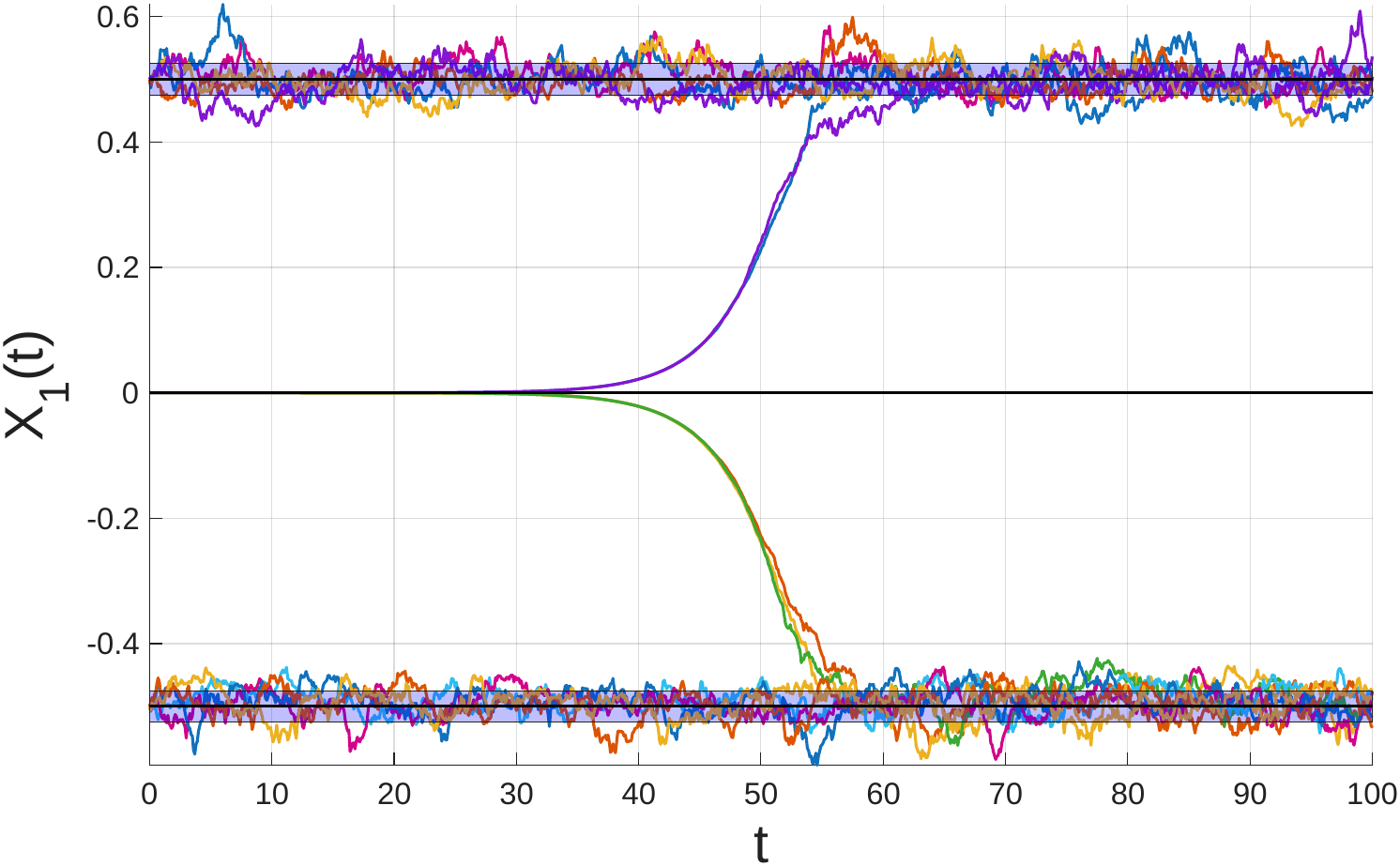}
        \caption{}
    \label{fig:pitchfork_quadB}
    \end{subfigure}
    \end{center}
    \caption{Pitchfork SDE with quadratic multiplicative noise with $\sigma=0.1$. (a) Bifurcation diagram and (b) sample trajectories at $\gamma=0.25$.}
    \label{fig:pitchfork_quad}
\end{figure}

\subsubsection{Fold bifurcation.}
\label{sec:fold}
We examine the following scalar SDE with either linear multiplicative noise or additive noise for which the corresponding ODE has a fold bifurcation at $\gamma=0$,
$$dX(t)  = \big[\gamma - X(t)^2\big] dt + \sigma_{11} X(t) dW_1(t) + \sigma_{12}dW_2(t).$$
Then for $\gamma>0$ there are two equilibria of the deterministic system, $\Xc_{\pm}=\pm\sqrt{\gamma}$.
We take $\gamma\geq 0$ as our bifurcation parameter.
First consider the case when $\sigma_{12}=0$ and set $\sigma_{11}=0.1$. In this case we note that for $t>0$, the solution $X(t)\geq 0$ if $X(0)=X_0\geq 0$ ($X(t)\leq 0$ for $X_0\leq 0$) and the SDE is mean-square dissipative in $\mathbb{R}^+$, but not in the whole of phase space $\mathbb{R}$. 
In Figure~\ref{fig:FoldMultBifA} we show the bifurcation diagram and in Figure~\ref{fig:FoldMultTrajB} five sample paths. 
\begin{figure}[htbp]
    \begin{center}
     \begin{subfigure}[b]{0.49\textwidth}
         \includegraphics[width=\textwidth]{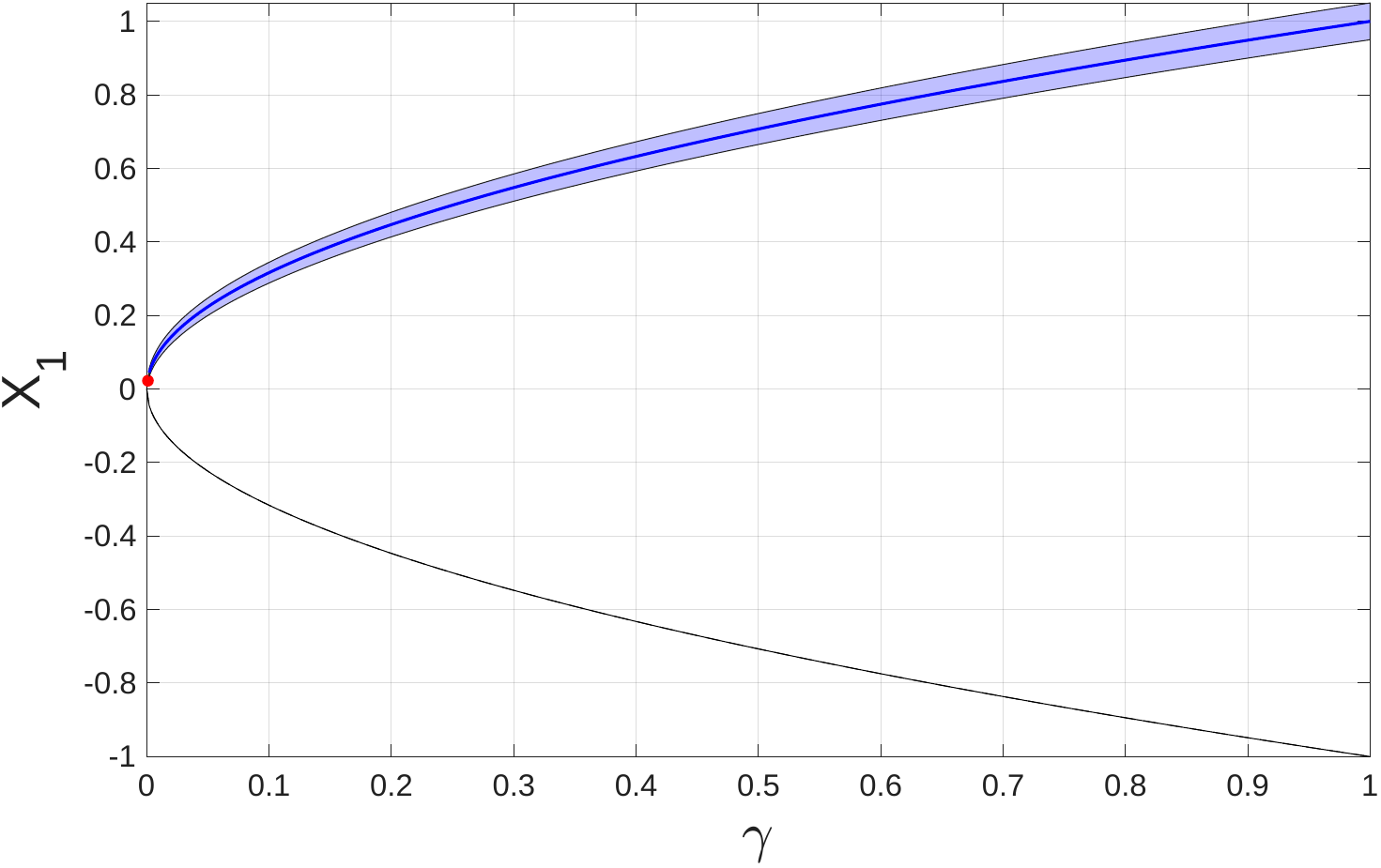}
    \caption{}
    \label{fig:FoldMultBifA}
    \end{subfigure}
    \hfill
    \begin{subfigure}[b]{0.49\textwidth}
    \includegraphics[width=\textwidth]{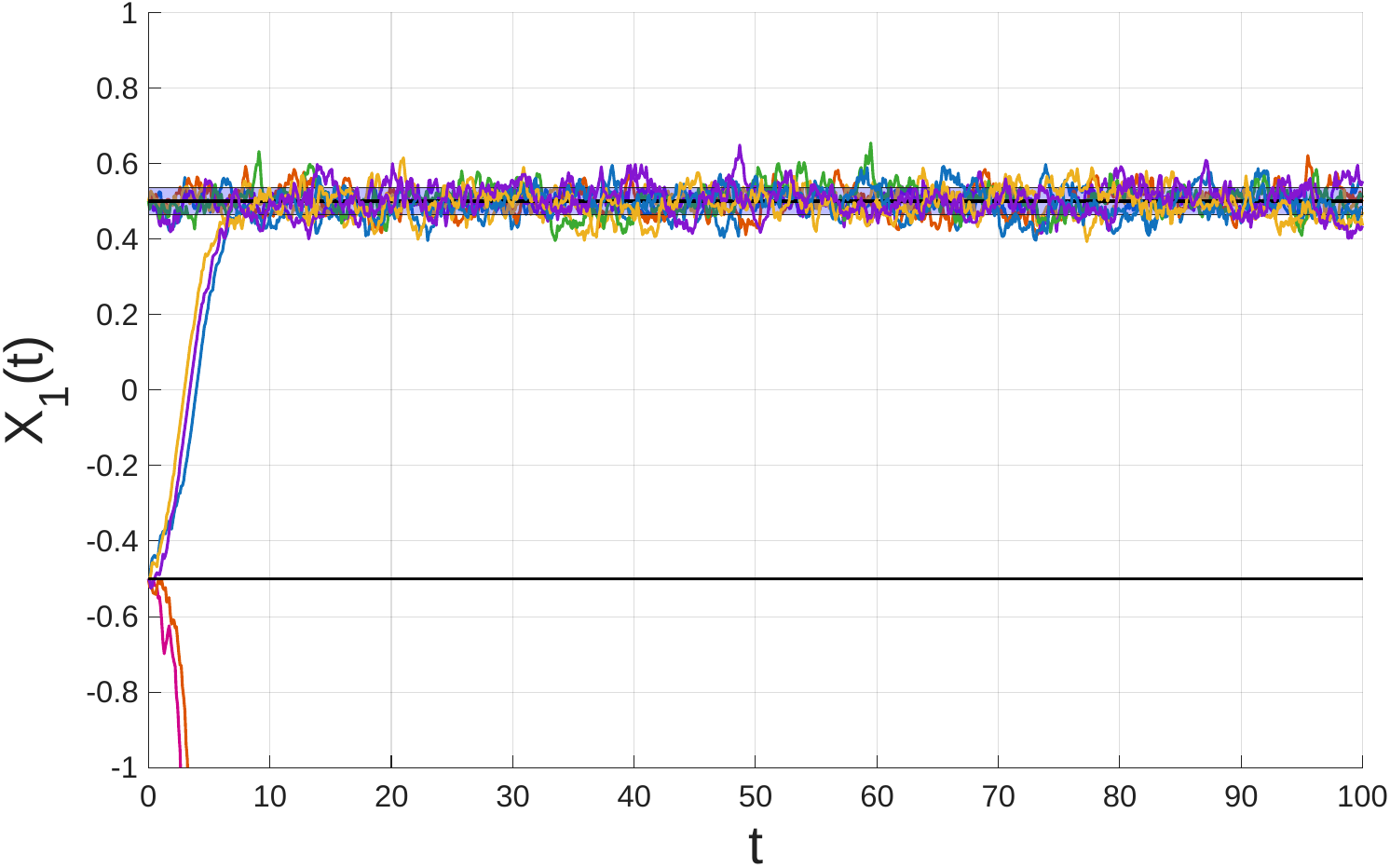}
        \caption{}
    \label{fig:FoldMultTrajB}
    \end{subfigure}
    \end{center}
    \caption{ (a) Bifurcation diagram for the fold SDE with multiplicative noise as $\gamma$ varies. This SDE is mean-square dissipative for $X(0)=X_0\geq 0$. (b) Sample trajectories at $\gamma=0.25$.}
    \label{fig:FoldMult}
\end{figure}

For the SDE with additive noise, i.e.~if $\sigma_{12}\neq 0$, the system is no longer mean-square dissipative, since solutions emanating from positive initial data may attain negative values and hence, the condition \eqref{eq:dissipative} in Assumption~\ref{ass:dissipative} no longer holds. We fix $\sigma_{12}=0.1$ and take $\sigma_{11}=0$. In Figure~\ref{fig:FoldAddBifA} we present the bifurcation diagram and in Figure~\ref{fig:FoldAddTrajB} we show five sample paths. Despite this, $\BETA$ gives some indication of the variability around $\Xc$ before trajectories diverge, in particular for those trajectories that remain positive for $t\in[0,100]$.
\begin{figure}[htbp]
    \begin{center}
     \begin{subfigure}[b]{0.49\textwidth}
         \includegraphics[width=\textwidth]{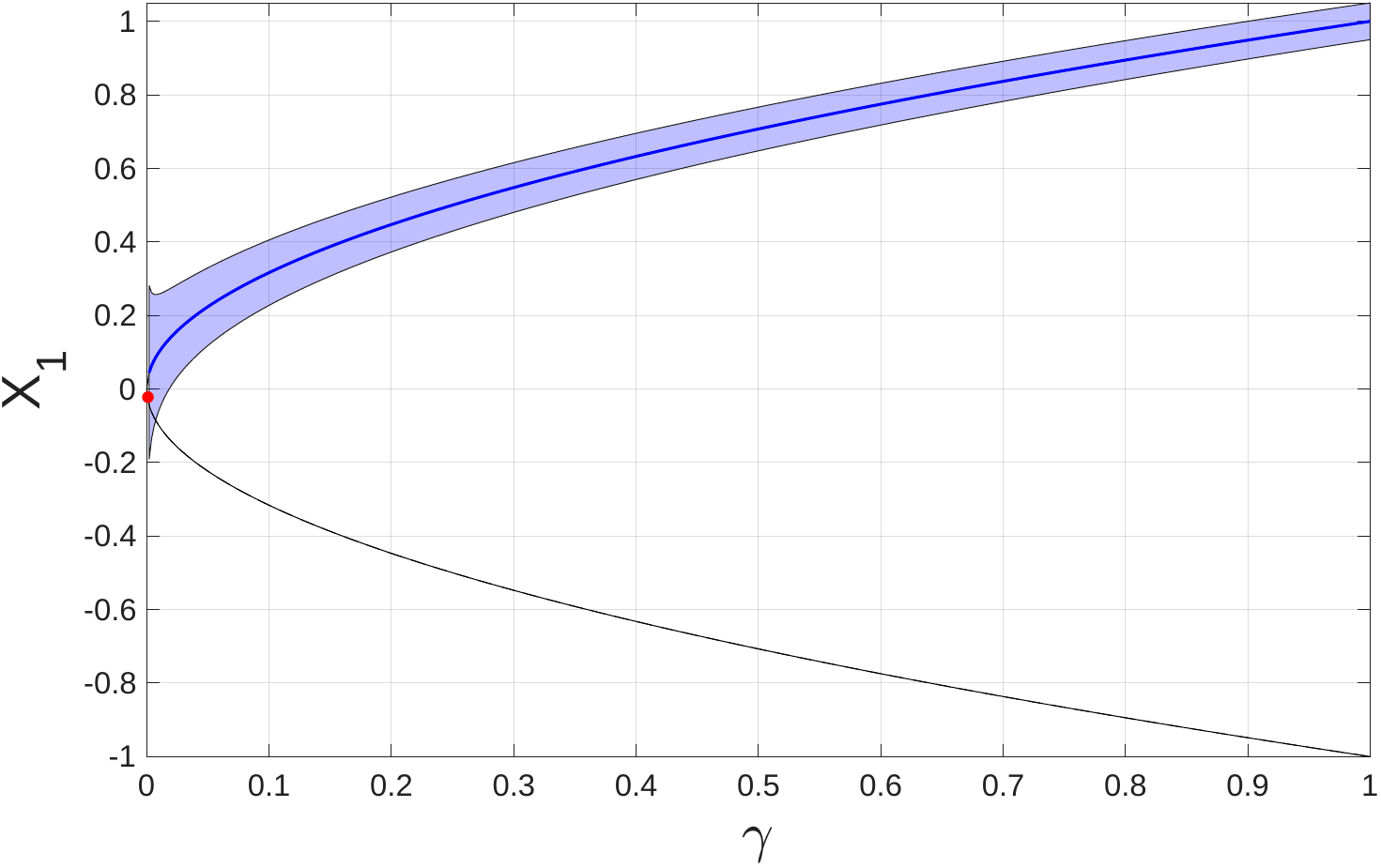}
    \caption{}
    \label{fig:FoldAddBifA}
    \end{subfigure}
    \hfill
    \begin{subfigure}[b]{0.49\textwidth}
    \includegraphics[width=\textwidth]{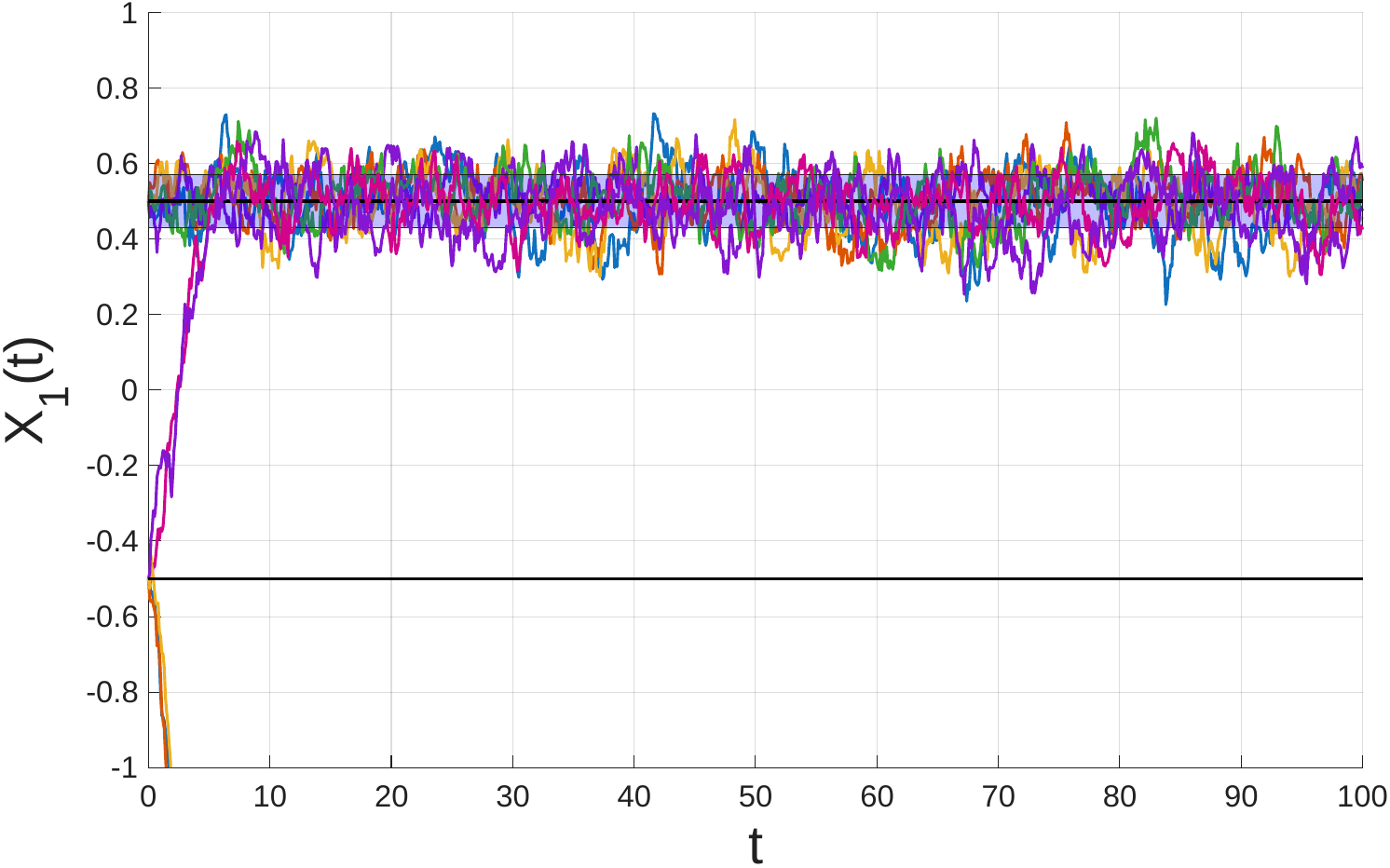}
        \caption{}
    \label{fig:FoldAddTrajB}
    \end{subfigure}
    \end{center}
    \caption{ (a) Bifurcation diagram for the fold SDE with additive noise as $\gamma$ varies. This SDE is not mean-square dissipative. (b) Sample trajectories at $\gamma=0.25$.}
    \label{fig:FoldAdd}
\end{figure}

\subsubsection{Transcritical bifurcation.}
\label{sec:trans}
We examine the following scalar SDE with either linear multiplicative noise or additive noise for which the corresponding ODE has a transcritical bifurcation at $\gamma=0$,
$$dX(t)  = \big[X(t)(\gamma - X(t))\big] dt + \sigma_{11} X(t) dW_1(t) + \sigma_{12}dW_2(t).$$
There are two equilibria of the deterministic system, $\Xc_{0}=0$ and $\Xc=\gamma$.
We take $\gamma$ as our bifurcation parameter.
First consider the case of multiplicative noise, i.e.~$\sigma_{12}=0$, and set $\sigma_{11}=0.1$. We note that then for $t>0$, $X(t)\geq 0$ if $X(0)=X_0\geq 0$ ($X(t)\leq 0$ for $X_0\leq 0$) and the SDE is mean-square dissipative for $X_0\geq 0$ for all $\gamma$. In Figure~\ref{fig:TransMultBifA} we show the bifurcation diagram and in Figure~\ref{fig:TransMultTrajB} five sample paths. 

\begin{figure}[htbp]
    \begin{center}
     \begin{subfigure}[b]{0.49\textwidth}
         \includegraphics[width=\textwidth]{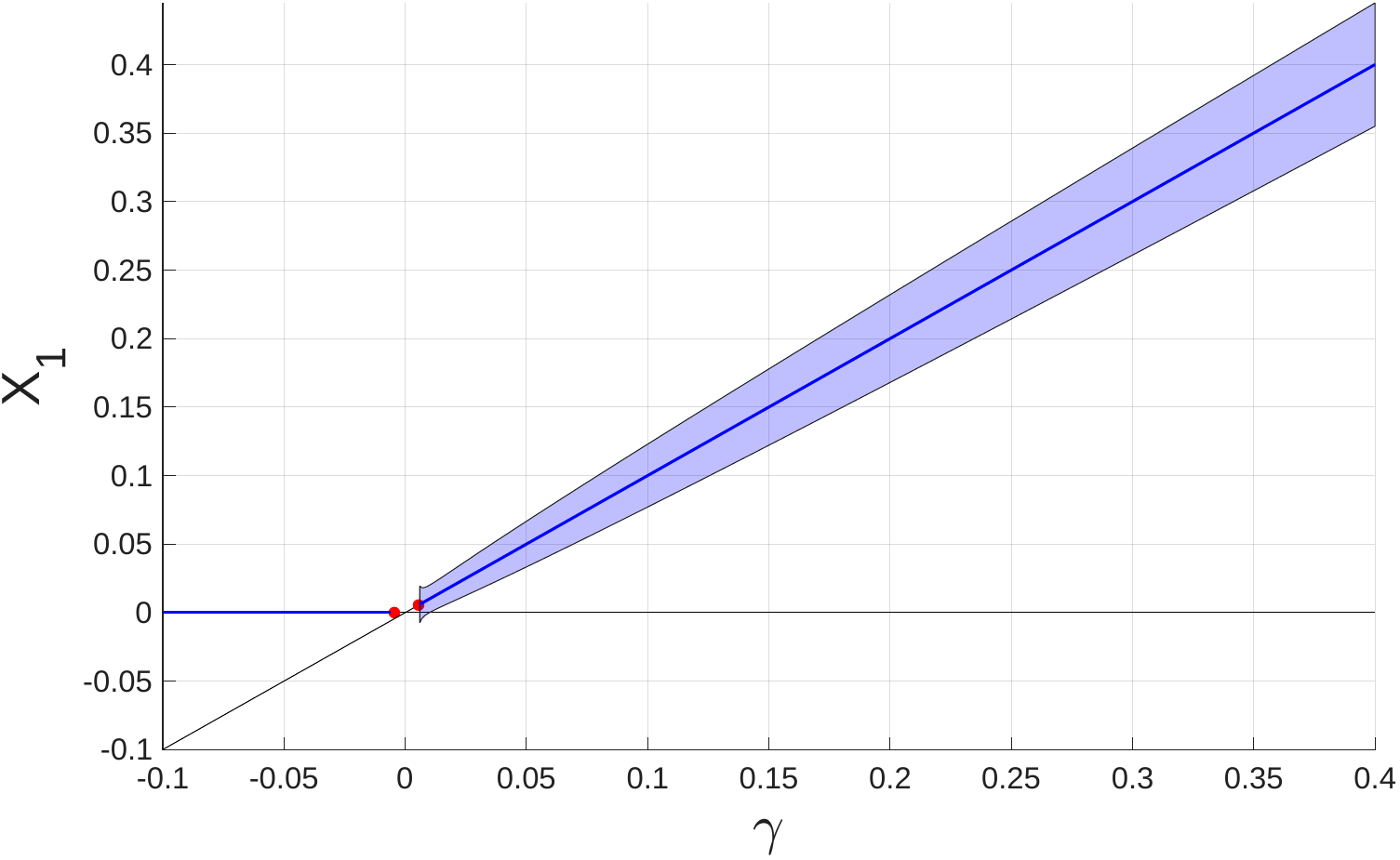}
    \caption{}
    \label{fig:TransMultBifA}
    \end{subfigure}
    \hfill
    \begin{subfigure}[b]{0.49\textwidth}
    \includegraphics[width=\textwidth]{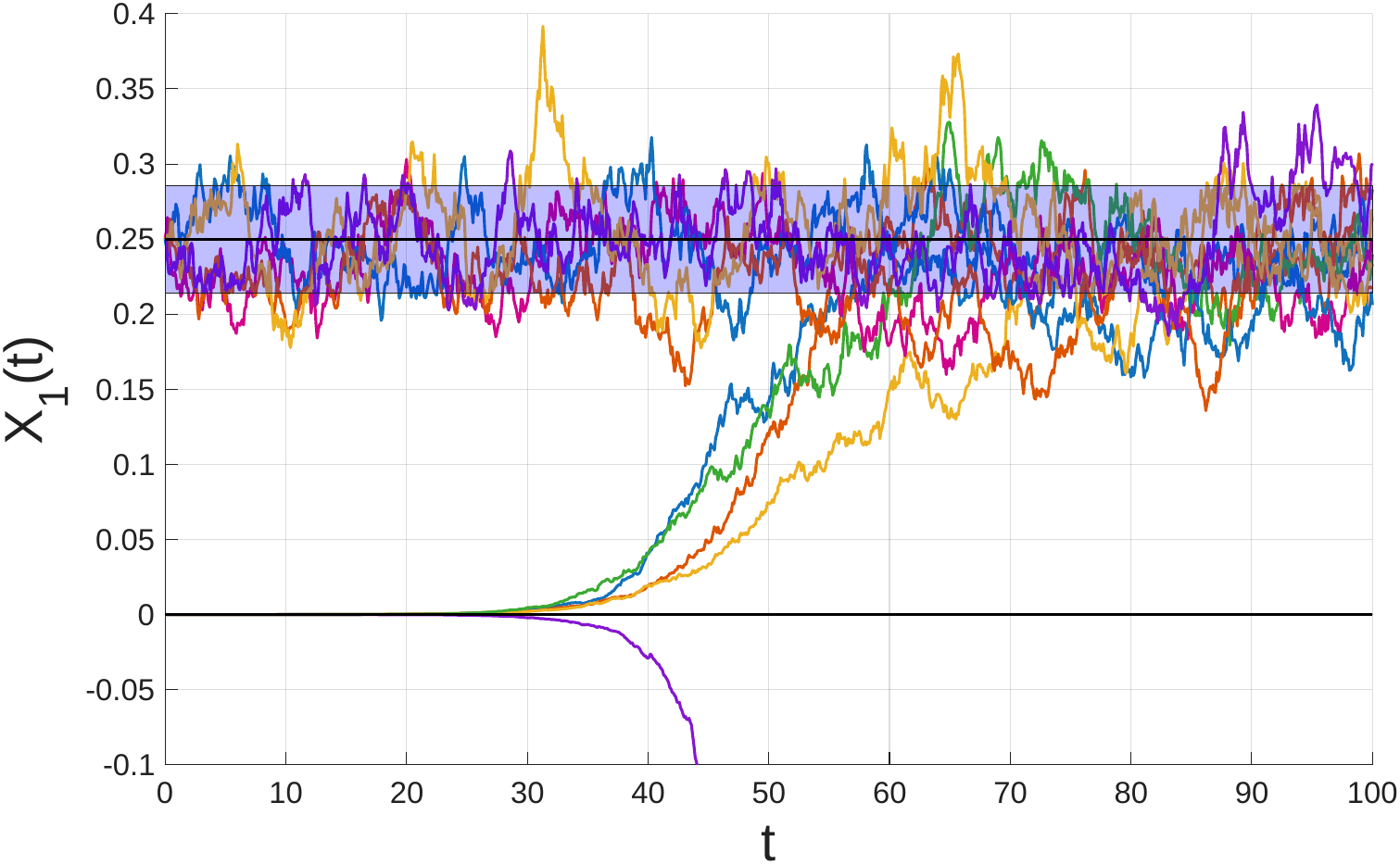}
        \caption{}
    \label{fig:TransMultTrajB}
    \end{subfigure}
    \end{center}
    \caption{ (a) Bifurcation diagram for the transcritical SDE with multiplicative noise as $\gamma$ varies. The system is mean-square dissipative for $X_0\geq 0$. (b) Sample trajectories at $\gamma=0.25$.}
    \label{fig:TransMult}
\end{figure}
In case of additive noise, i.e.~if $\sigma_{12}\neq 0$, the system is no longer mean-square dissipative for $\gamma\geq 0$, for the same reason as for the fold bifurcation example. We fix $\sigma_{12}=0.1$ and take $\sigma_{11}=0$. In Figure~\ref{fig:TransAddBifA} we present the bifurcation diagram and in Figure~\ref{fig:TransAddTrajB} we show five sample paths. Once again, $\BETA$ gives some indication of the variability around $\Xc$ before trajectories diverge.
\begin{figure}[htbp]
    \begin{center}
     \begin{subfigure}[b]{0.49\textwidth}
         \includegraphics[width=\textwidth]{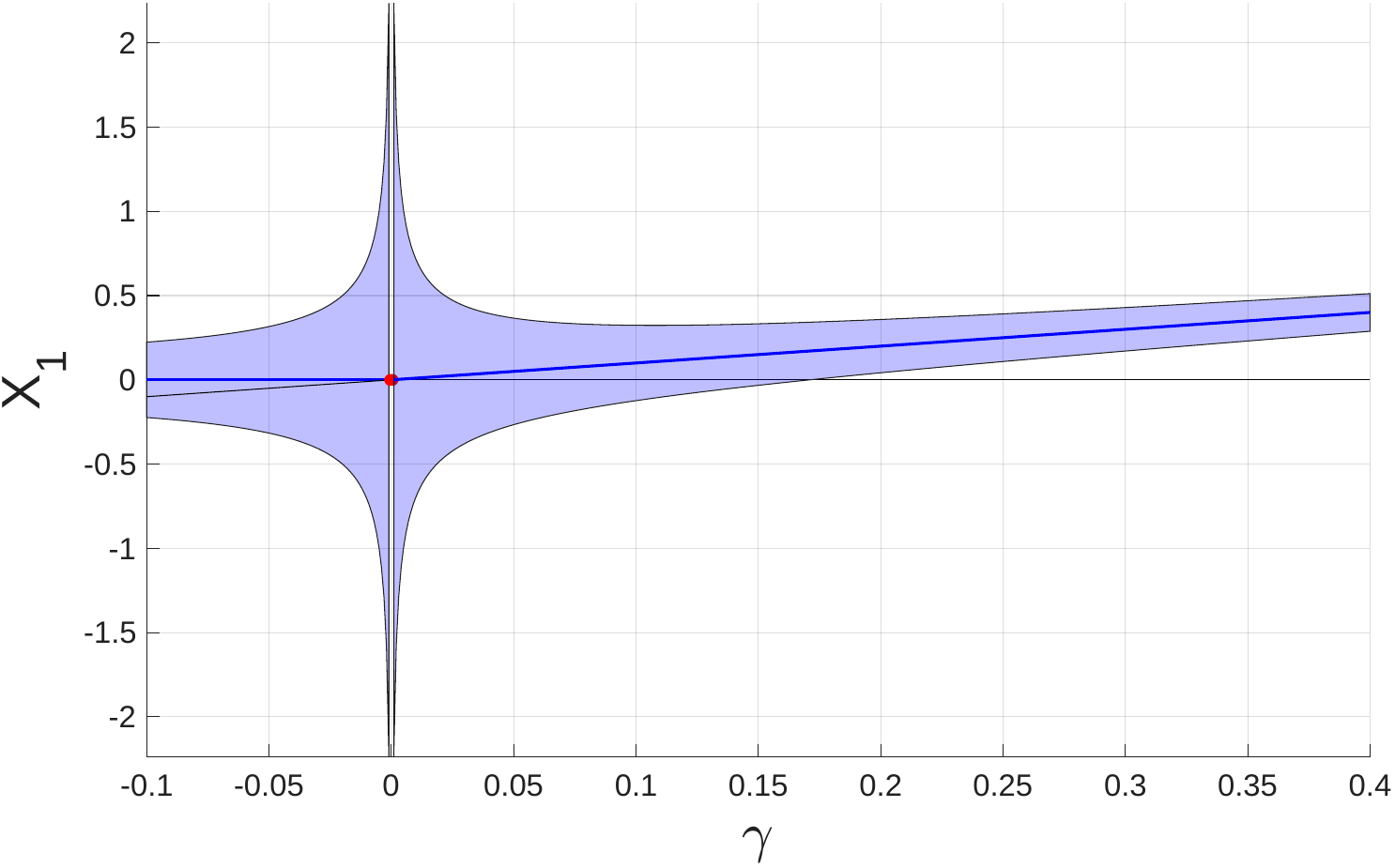}
    \caption{}
    \label{fig:TransAddBifA}
    \end{subfigure}
    \hfill
    \begin{subfigure}[b]{0.49\textwidth}
    \includegraphics[width=\textwidth]{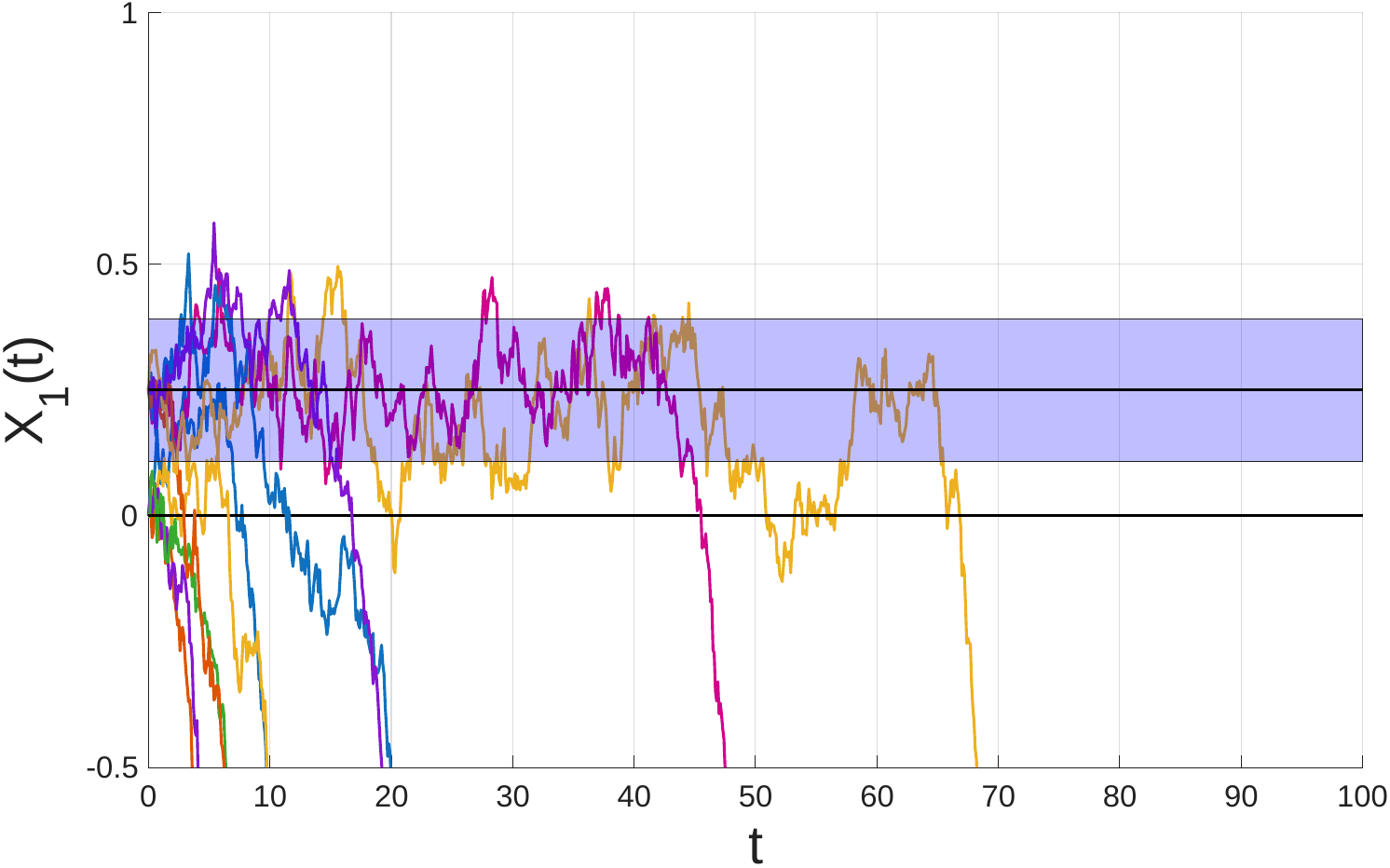}
        \caption{}
    \label{fig:TransAddTrajB}
    \end{subfigure}
    \end{center}
    \caption{ (a) Bifurcation diagram for the transcritical SDE with additive noise as $\gamma$ varies. Note that the system is not  mean-square dissipative. (b) Sample trajectories at $\gamma=0.25$.}
    \label{fig:TransAdd}
\end{figure}
\subsubsection{The Cox-Ingersoll-Ross (CIR) model.}
The CIR model is used prominently in finance as a model for interest rates and stochastic volatility processes. It is given by
\begin{equation}\label{eq:CIR}
d X(t) = \kappa(\theta - X(t)) dt + \sigma\sqrt{X(t)} dW(t);\quad X(0)=X_0>0,
\end{equation}
where $\kappa,\theta,\sigma>0$. We fix here $\kappa=2$ and $\theta=0.02$ and vary $\sigma$. The corresponding deterministic system (with $\sigma=0$) has a single equilibrium $\Xc=\theta$. It is known that solutions of \eqref{eq:CIR} are almost surely (a.s.)~non-negative for all parameter values, and a.s.~positive if and only if $2\kappa\theta>\sigma^2$. This inequality is known as `Feller's condition': if it is violated then trajectories can achieve the boundary at zero.  It is straightforward to show that \eqref{eq:CIR} is mean-square dissipative however, unless the Feller condition holds, the diffusion coefficient $G(x)=\sigma\sqrt{x}$ does not satisfy \eqref{eq:D^2f+D^2g}.
It is known (see~\cite{CIR}) that $\expect{X(t)}=X_0e^{-\kappa t}+\theta(1-e^{-\kappa t})$, and therefore the solution is mean-reverting: $\lim_{t\to\infty} \expect{X(t)} = \theta$, with a timescale for that reversion determined by $\kappa$. Note also that $\lim_{t\to\infty}\text{Var}[X(t)]=\theta\sigma^2/2\kappa$. 

The mean-square nonlinear stability of $\Xc$  is determined from \eqref{eq:MU} and   $\MU= 0$ when $-8\kappa\theta+4\kappa\delta_1+(1+\delta_2)\sigma^2=0$ for arbitrarily small $\delta_1$, $\delta_2$. 
This means for $\sigma>2\sqrt{2\kappa\theta}$ the deterministic equilibrium $\Xc$ does not satisfy the condition of non-linear mean-square stability in Definition \ref{Def:nonlinear-m-s-stability}, but is stable in mean. See Figure~\ref{fig:CIRBifA} for the bifurcation diagram and Figure~\ref{fig:CIRTrajB} for sample trajectories where $\sigma=2\sqrt{2\kappa\theta}$ (where the positive real parts of the numerical approximation are plotted). 

\begin{figure}[htbp]
    \begin{center}
     \begin{subfigure}[b]{0.49\textwidth}
         \includegraphics[width=\textwidth]{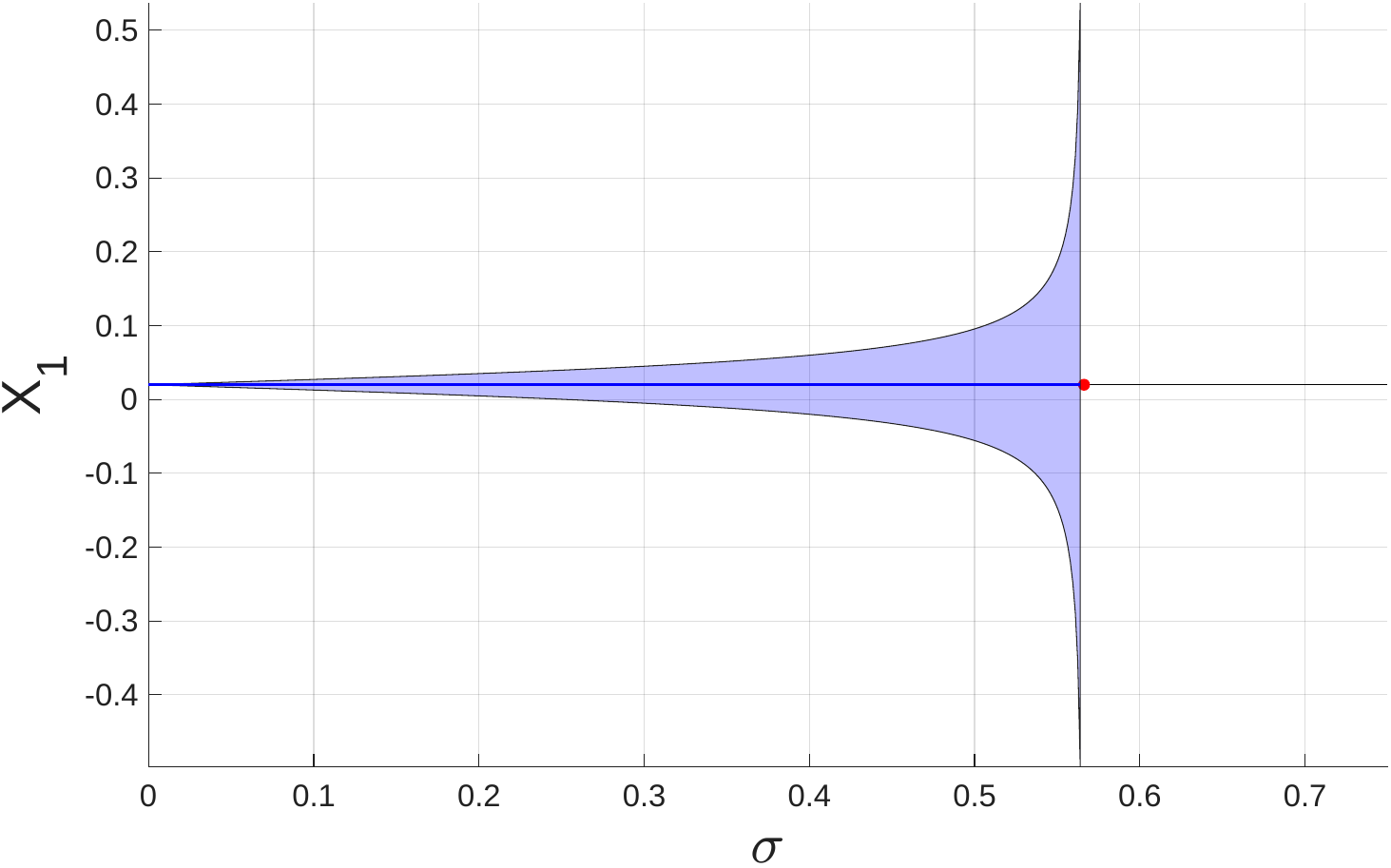}
    \caption{}
    \label{fig:CIRBifA}
    \end{subfigure}
    \hfill
    \begin{subfigure}[b]{0.49\textwidth}
    \includegraphics[width=\textwidth]{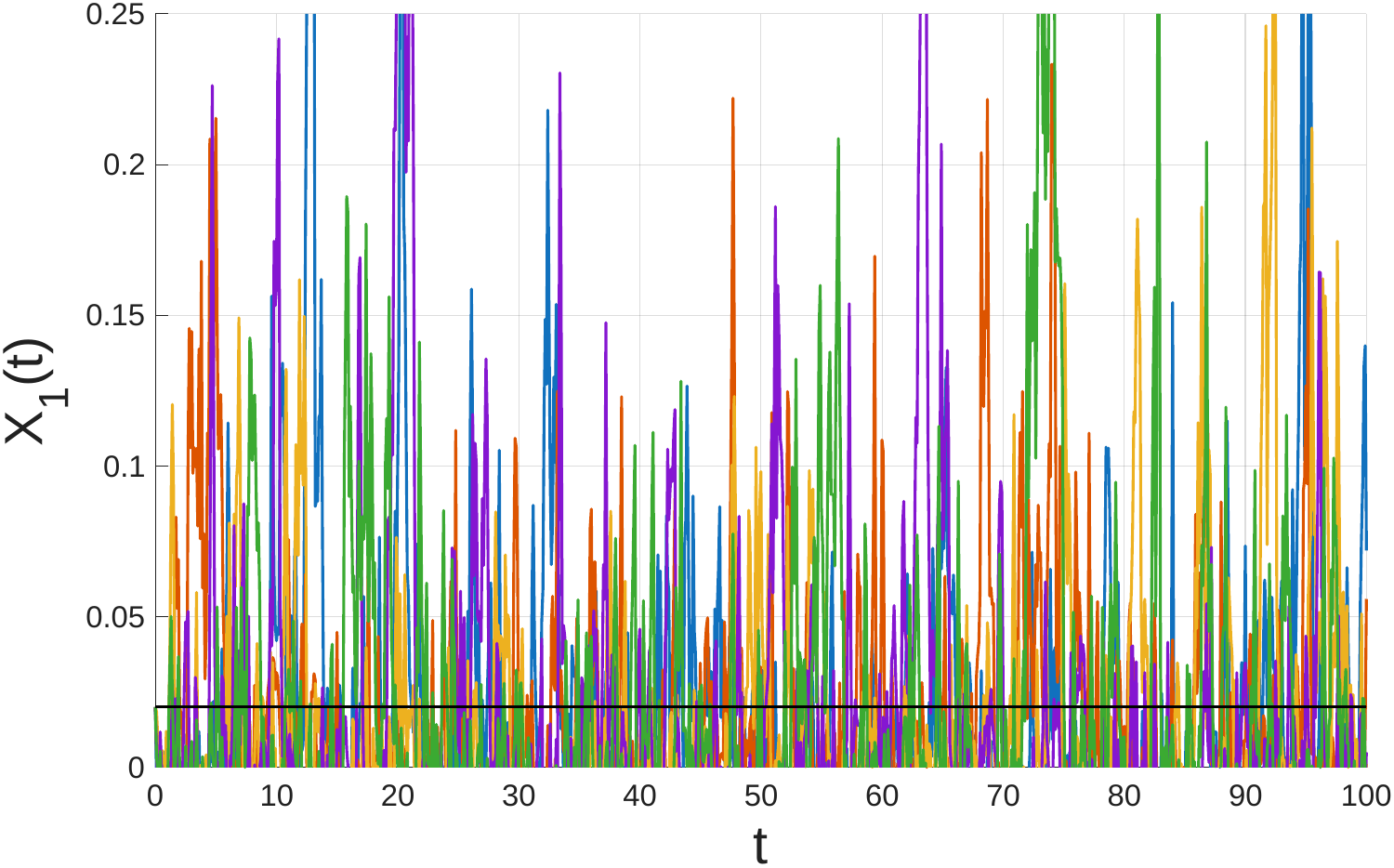}
        \caption{}
    \label{fig:CIRTrajB}
    \end{subfigure}
    \end{center}
    \caption{ (a) Bifurcation diagram for the CIR SDE \eqref{eq:CIR} as $\sigma$ varies. (b) Sample trajectories (real parts) at $\sigma=4\sqrt{0.02}$ where $\Xc=\theta$ stops being mean-square nonlinearly stable.}
    \label{fig:CIR}
\end{figure}

\subsection{Bistability example.}
\label{sec:BiStable}
We now consider the SDE discussed in Section~\ref{sec:intro}. 
The deterministic system in $d=2$ exhibits bistability and hysteresis and we consider multiplicative noise with $m=3$. 
\begin{align*}
    d X_1(t) &= \left[\gamma + X_1(t) - (X_1(t))^3\right] dt + \sigma_{11} X_1(t) dW_1(t) 
    + \sigma_{13} X_1(t) X_2(t) dW_3(t),\\ 
  d X_2(t) &= -X_2(t) dt 
  +  \sigma_{22} X_2(t) dW_2(t) + \sigma_{23} X_1(t)X_2(t)dW_3(t).
\end{align*}
We take $\sigma_{11}=0.25$, $\sigma_{13}=0.01$, $\sigma_{22}=0.1$, and $\sigma_{23}=0.01$. 
We plot the first component $X_1$ as $\gamma$ is varied in 
Figure~\ref{fig:BiStabMultA} and in Figure~\ref{fig:BiStabMultB}
we illustrate some sample paths. As already noted it is evident that the mean-square bifurcations for the stochastic system do not occur at the deterministic fold points and that we observe transitions between the two deterministic equilibria.
Similar bifurcation diagrams and sample paths are observed in the case of additive noise for this example. 

\subsection{Lorenz equations.}
\label{sec:lorenz}
We consider the classic Lorenz equations ($d=3$) with multiplicative noise ($m=3$),
\begin{equation*}
\begin{aligned}
    dx(t) & = s\big(y(t)-x(t)\big)dt + \sigma_{11}x(t) dW_1(t),\\
    dy(t) & = \Big[x(t)\big(\rho-z(t)\big)-y(t)\Big]dt +\sigma_{23}x(t)z(t) dW_3(t) +  \sigma_{22}y(t) dW_2(t), \\
    dz(t) & = \big[x(t)y(t)-b z(t)\big]dt +\sigma_{32}x(t)y(t) dW_{2}(t)+  \sigma_{33} z(t) dW_3(t).
\end{aligned}
\end{equation*}
We consider $\rho$ as the bifurcation parameter and fix $b=8/3$ and $s=10$ (the typical values of the parameters).
Recall that the equilibria of  the deterministic Lorenz system are $(x^\ast, y^\ast, z^\ast) = (0,0,0)$ for all $\rho\geq 0$, and  $(x^\ast, y^\ast, z^\ast) = (\pm\kappa, \pm\kappa, \rho -1 )$ for $\rho >1$ 
where $\kappa = \sqrt{b(\rho-1)}$.
Typically the `chaotic' or `strange attractor' is seen for $\rho>28$ for the standard parameter values.

First we show that the SDE is mean-square dissipative. Using the notation $X=(x,y,z)$ we have
$$
F(X) = \begin{pmatrix} 
s(y-x) \\
x(\rho - z) - y \\
xy - b z
\end{pmatrix}, 
\quad 
G(X) = \begin{pmatrix} 
\sigma_{11} x & 0 &0 \\
0& \sigma_{22} y & \sigma_{23} xz \\
0 &  \sigma_{32} xy & \sigma_{33} z 
\end{pmatrix} .
$$
The Lyapunov function for the (deterministic) Lorenz system reads 
$V(x,y,z) = \rho x^2 + s y^2 + s(z-2\rho)^2=:\|\widetilde X\|^2$. Note that moment bounds for $\tilde X= (\sqrt{\rho} x, \sqrt{s} y, \sqrt{s} (z-2\rho))^T$ imply moment bounds for $X$ since $\tilde X$ is only a scaling and constant shift of $X$. If we were to apply the It\^o formula to $V$ as in the proof of Lemma~\ref{lem:dissipativemomentbound} (see \eqref{eq:forlorenz}) then, rather than \eqref{ass:dissipative}, we would encounter
\begin{equation*} 
\begin{aligned}
&\frac 12 \langle DV(X), F(X) \rangle + \frac 1 4|{\rm Tr } (G(X)^T  D^2 V(X) G(X))| = 
- s\rho x^2 + s \rho yx + s \rho xy - s xzy - s y^2 \\& \quad+ s xy(z-2\rho)- s b z(z-2\rho)  + 
\frac 12\big[\rho\sigma_{11}^2 x^2 + s\sigma_{22}^2 y^2+  s\sigma_{33}^2 z^2
+ s x^2 (\sigma_{23}^2 z^2 + \sigma_{32}^2 y^2)\big]\\
& = -\rho(s- \sigma_{11}^2/2 ) x^2 
- s(1-  \sigma_{22}^2/2 - \sigma_{32}^2 x^2/2) y^2 -
s(b - \sigma_{33}^2/2- \sigma_{23}^2 x^2/2)\Big[(z-2\rho)^2
\\
& \quad +2\rho\Big(\frac{(b -  \sigma_{33}^2 -  \sigma_{23}^2 x^2)}{(b - \sigma_{33}^2/2- \sigma_{23}^2 x^2/2)} z - 2\rho \Big)\Big].
\end{aligned}
\end{equation*}
For $\sigma_{23}=\sigma_{32}=0$ and appropriately small $\sigma_{ii}$, $i=1,2,3$, we obtain 
$$
\frac 12 \langle DV(X), F(X) \rangle + \frac 1 4|{\rm Tr } (G(X)^T  D^2 V(X) G(X))|  \leq - \alpha_2 V(X) + \alpha_3,
$$
for $\alpha_3 \geq  4\rho^2\big[  (b- \sigma_{33}^2)^2  /[(2 b - \sigma_{33}^2)^2(1-\tilde \alpha_2)] +  \sigma^2_{33}/(2b-\sigma_{33}^2)\big]$ and  $0<\alpha_2 = \min\{(s - \sigma_{11}^2/2), (1 - \sigma_{22}^2/2), \tilde \alpha_2(b - \sigma_{33}^2/2)\}$, where $0< \tilde \alpha_2<1$.

If $\sigma_{23}$ and $\sigma_{32}$ are non-zero, for initial conditions for $x$ and sufficiently small $\sigma_{23}$ and $\sigma_{32}$ satisfying 
\begin{equation}\label{eq:lorenzDissCond}
|x|^2 \leq  \kappa  \quad \text{ and } \quad   \rho^2 \leq  \kappa \alpha_2 (1- \tilde \alpha_2) (e/4),
\end{equation}
where $\kappa = \min\{(1 - \sigma^2_{22})/\sigma^2_{32}, ( b -  \sigma_{33}^2)/\sigma_{23}^2 \}$, 
we obtain the dissipativity condition. 

In our numerical experiments we take 
$\sigma_{11}=\sigma_{22}=\sigma_{33}=0.01$ and 
 for linear diagonal noise we have $\sigma_{23}=\sigma_{32}=0 $. In Figure~\ref{fig:LorenzBif} we plot the bifurcation diagrams for the linear diagonal noise in~(a) and non-linear noise in~(b) where $\sigma_{23}=\sigma_{32}=0.01$ showing $x$ as $\rho$ is varied. We observe that with the nonlinear noise $\BETA$ is larger, in particular as $\rho$ increased.
Figure~\ref{fig:LorenzTraj} shows sample trajectories of $x(t)$ for $t\in[0,100]$. For $\rho=10$ we observe the sample trajectories remain close to the nonlinearly mean-square stable equilibria. However for larger $\rho=20$
we see emergence of oscillations. Although equilibria are mean-square stable we observe transitions to this other type of stable object (that is likely to be a stochastic periodic orbit).
\begin{figure}[htbp]
    \begin{center}
     \begin{subfigure}[b]{0.49\textwidth}
         \includegraphics[width=\textwidth]{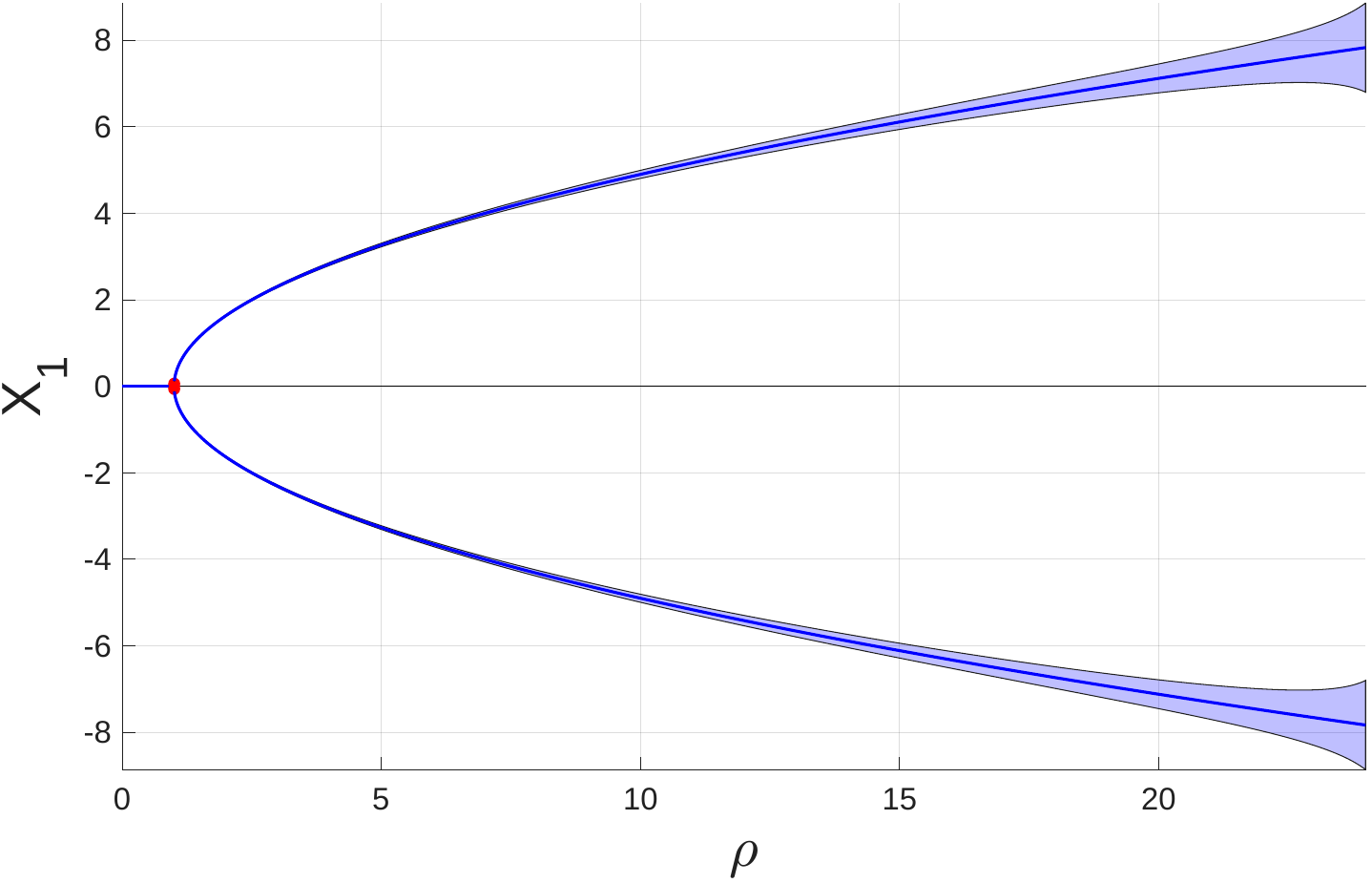}
    \caption{}
    \label{fig:LorenzBifA}
    \end{subfigure}
    \hfill
    \begin{subfigure}[b]{0.49\textwidth}
    \includegraphics[width=\textwidth]{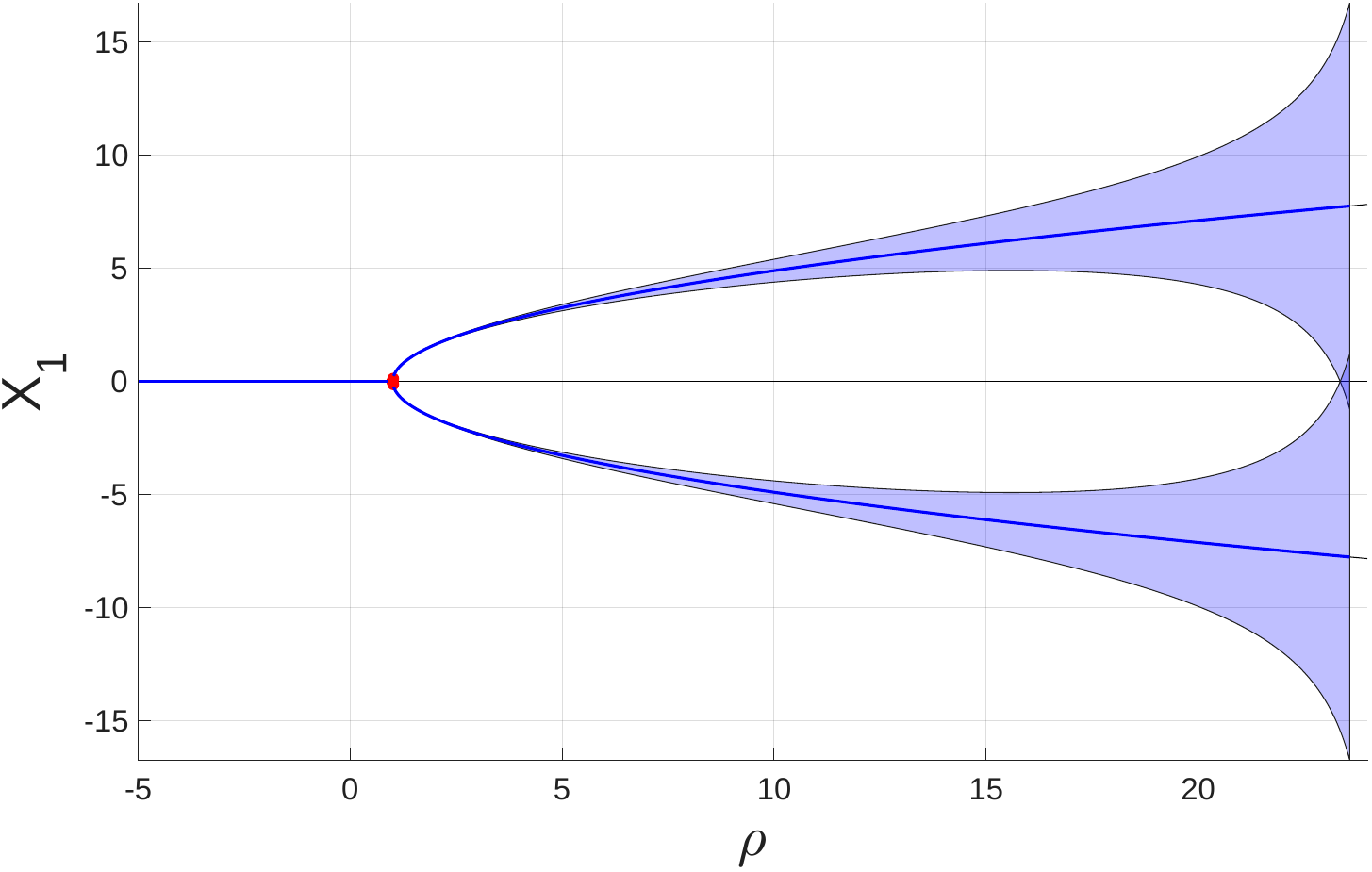}
        \caption{}
    \label{fig:LorenzBifB}
    \end{subfigure}
    \end{center}
    \caption{Bifurcation diagrams for the Lorenz equations as $\rho$ varies (a) diagonal linear multiplicative noise and (b) non-linear multiplicative noise.}
    \label{fig:LorenzBif}
\end{figure}
\begin{figure}[htbp]
    \begin{center}
     \begin{subfigure}[b]{0.49\textwidth}
         \includegraphics[width=\textwidth]{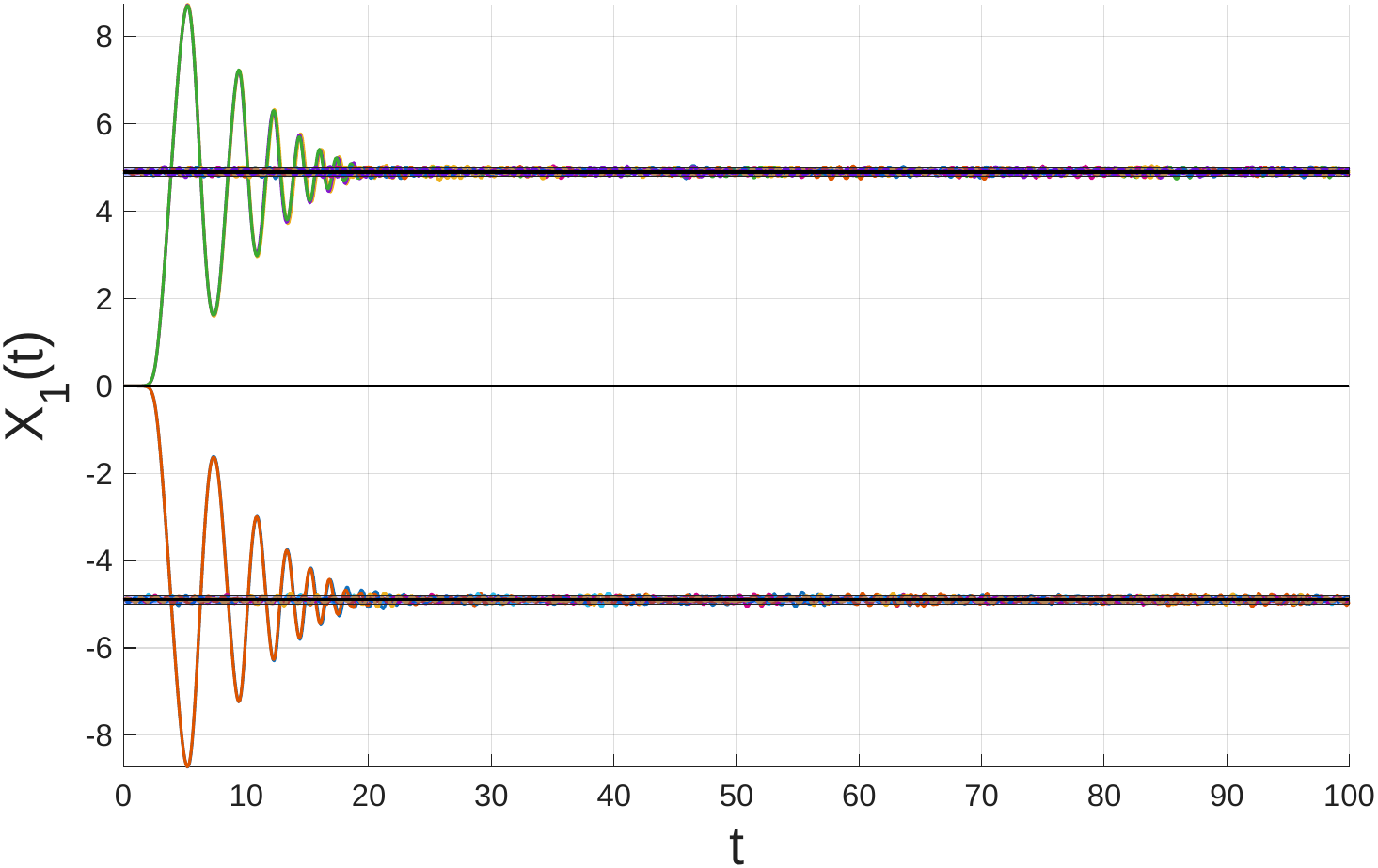} 
    \caption{$\rho=10$}
    \label{fig:LorenzTrajA}
    \end{subfigure}
    \hfill
    \begin{subfigure}[b]{0.49\textwidth}
        \includegraphics[width=\textwidth]{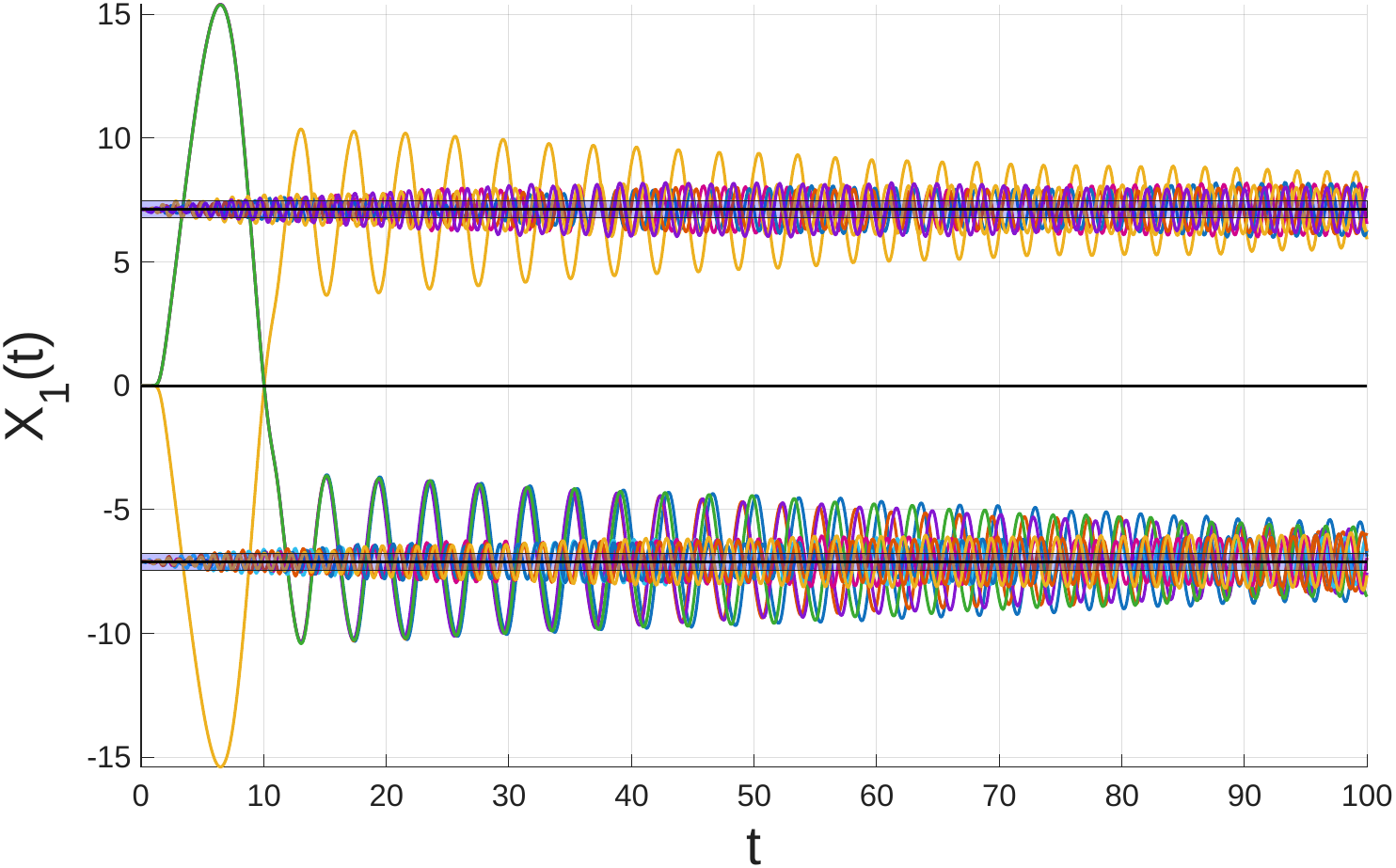}
        \caption{$\rho=20$}
    \label{fig:LorenzTrajB}
    \end{subfigure}
    \begin{subfigure}[b]{0.49\textwidth}
    \includegraphics[width=\textwidth]{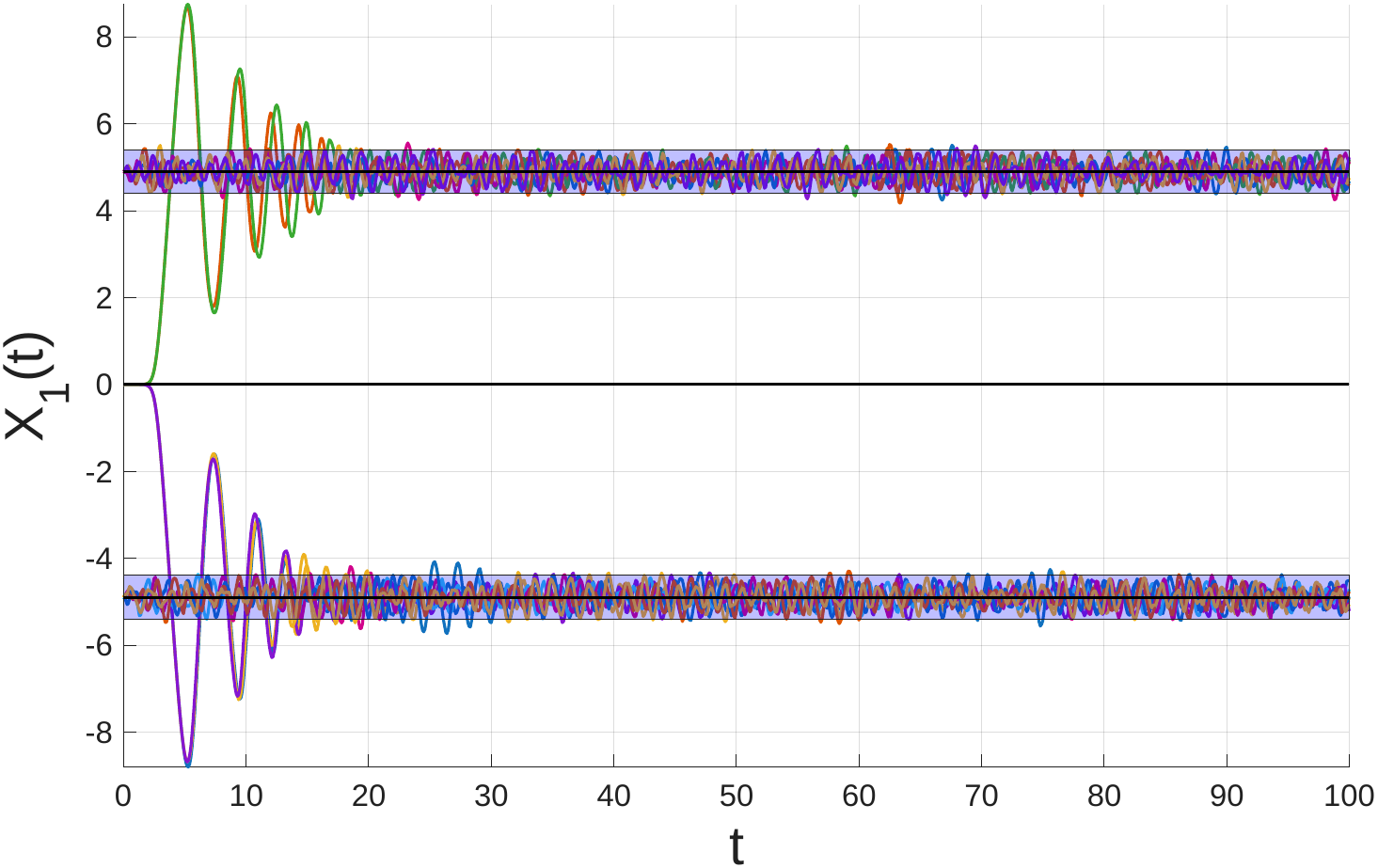} 
    \label{fig:LorenzTrajC}
            \caption{$\rho=10$}
    \end{subfigure}
        \begin{subfigure}[b]{0.49\textwidth}    
    \includegraphics[width=\textwidth]{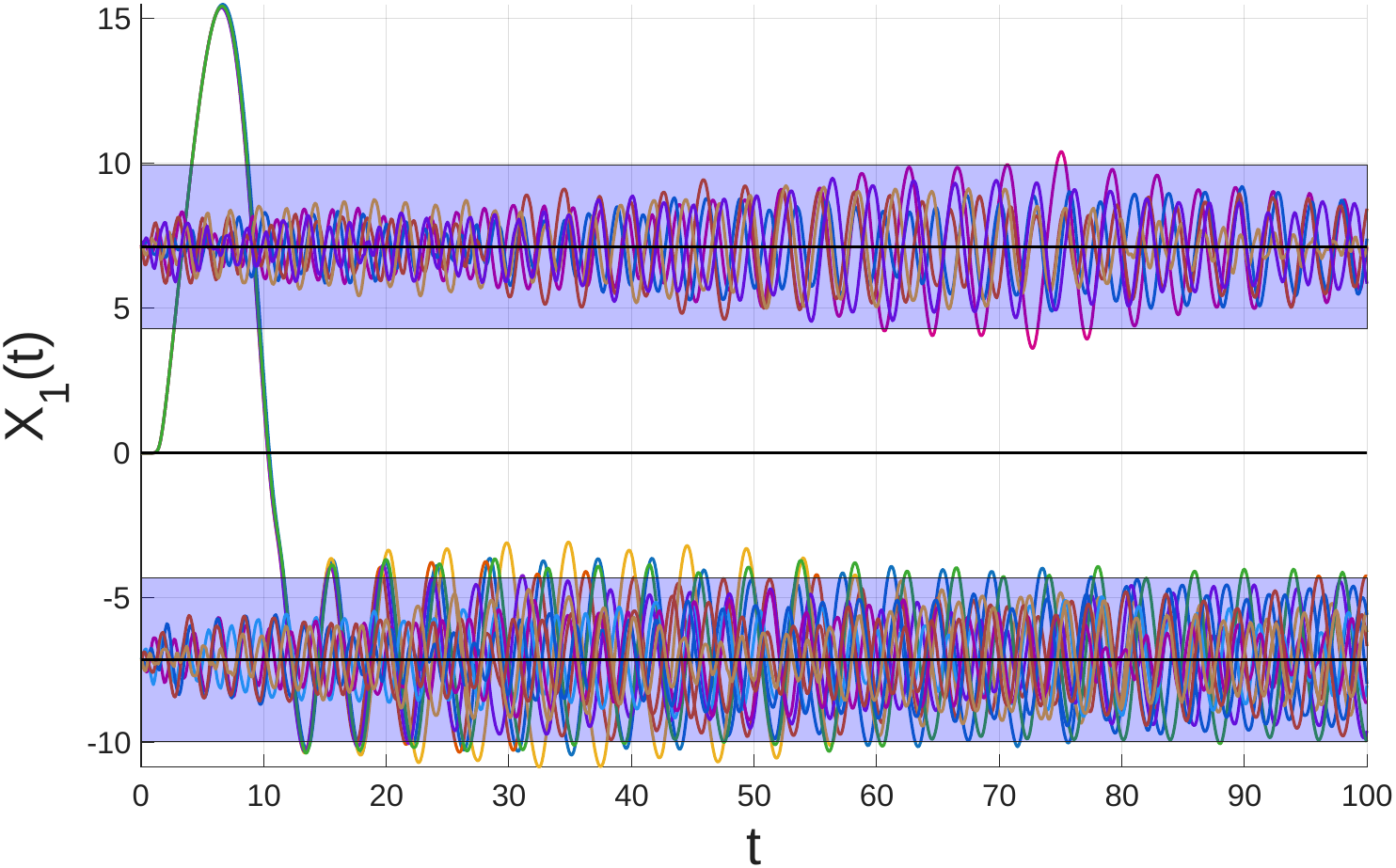}
    \label{fig:LorenzTrajD}
            \caption{$\rho=20$}
    \end{subfigure}
    \end{center}
    \caption{Sample paths for the Lorenz equations showing $x(t)$ for diagonal linear multiplicative noise (a), (b) and non-linear multiplicative noise (c), (d).}
    \label{fig:LorenzTraj}
\end{figure}

\subsection{Allen-Cahn system of SDEs.}
\label{sec:Allen-Cahn}

We consider a large system of SDEs $(d=m=50)$ that originates from a standard finite difference approximation in space of a stochastic semilinear heat equation with periodic boundary conditions and with space-time white noise, see for example \cite{LordPowellShardlow}.
That is we set $\Delta x=2\pi/(d-1)$ and consider the system of SDEs 
$$dX(t)= \big[ AX(t) + \gamma X(t) - (X(t))^3 \big] dt + \sigma X(t) dW(t), $$
where $\sigma = \overline{\sigma}/\sqrt{\Delta x}$, with $\overline{\sigma}=0.1$, and  
$$
A = \frac{1}{\Delta x^2}\begin{pmatrix}
   -2 & 1 & 0 & \ldots & & 1 \\
    1 & -2 & 1 & 0 & \ldots  & 0 \\
    0 & 1 & -2 & 1 & \ldots  & 0 \\
     & & \ddots & \ddots & \ddots & \\
    0 & 0 & \ldots & 1 &  -2& 1 \\
    1 & 0  & \ldots &0  & 1 & -2  
\end{pmatrix}.
$$
Dissipativity follows similarly as for the pitchfork bifurcation.
We plot in Figure~\ref{fig:ACBifA} the bifurcation diagram for this Allen-Cahn SDE, showing $\|X\|$ as $\gamma$ is varied.  Figure~\ref{fig:ACBifB} illustrates the evolution of $\|X(t)\|$ for $t\in[0,T]$ for some sample trajectories. We note how well $\BETA$ captures the observed variations in these samples.

\begin{figure}[htbp]
    \begin{center}
     \begin{subfigure}[b]{0.49\textwidth}
         \includegraphics[width=\textwidth]{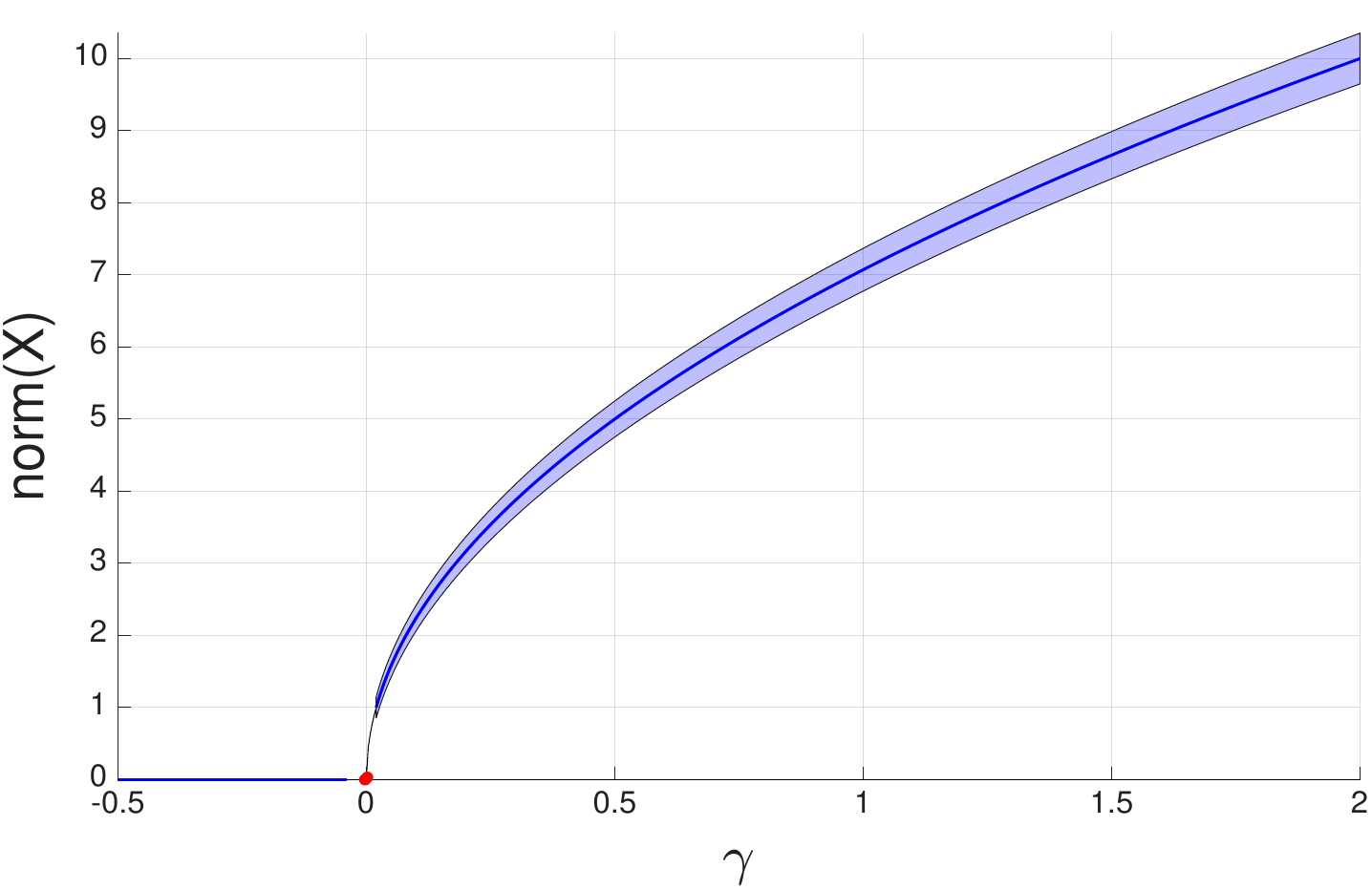}
    \caption{}
    \label{fig:ACBifA}
    \end{subfigure}
    \hfill
    \begin{subfigure}[b]{0.49\textwidth}
    \includegraphics[width=\textwidth]{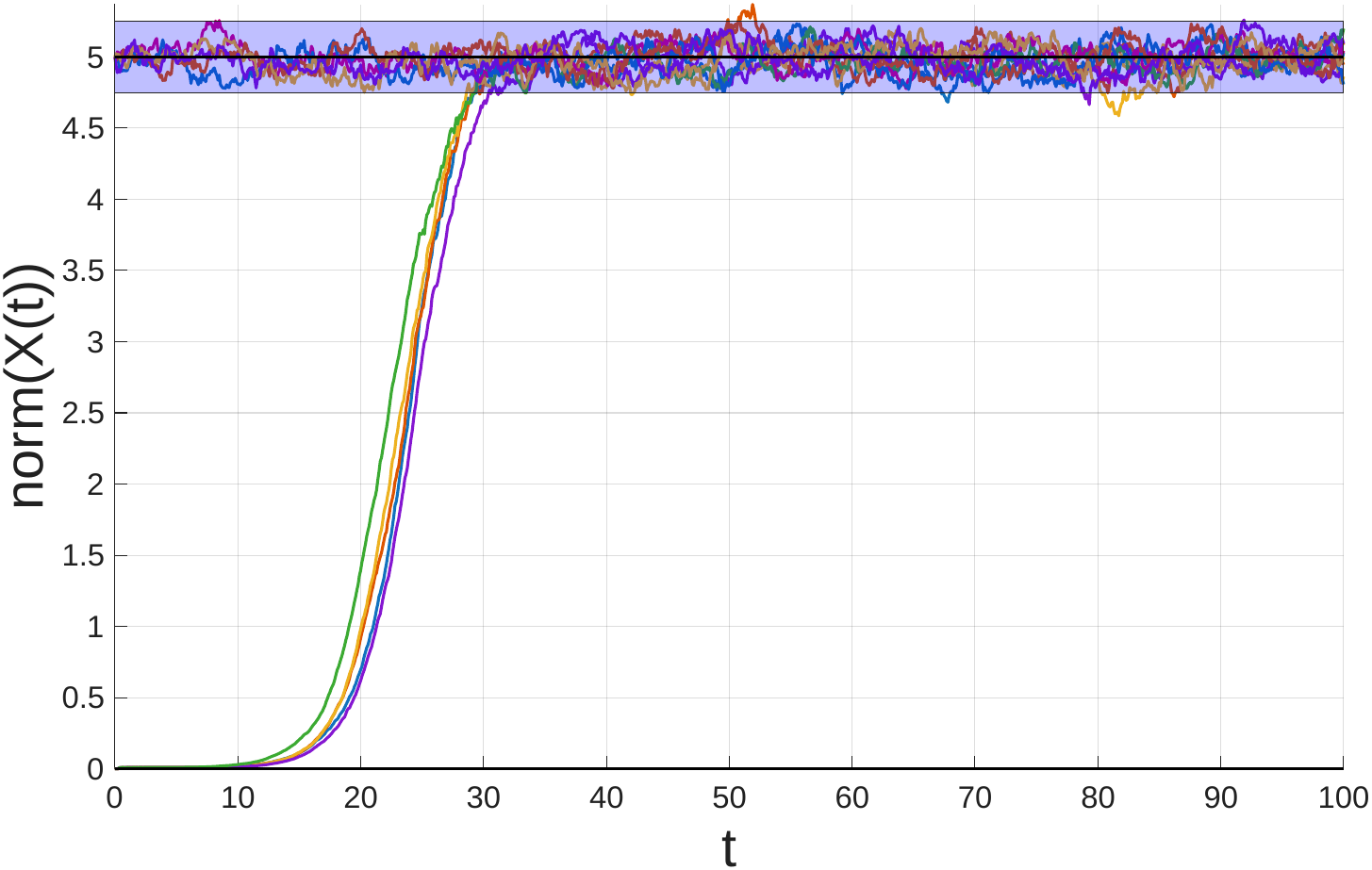}
        \caption{}
    \label{fig:ACBifB}
    \end{subfigure}
    \end{center}
    \caption{ (a) Bifurcation diagram for the Allen-Cahn SDE as $\gamma$ varies. (b) Norm of sample trajectories at $\gamma=0.5$.}
    \label{fig:AC}
\end{figure}

\section{Discussion and conclusion.}
We have developed and illustrated the use of a new mean-square based approach to examine the effect of stochastic forcing on the stability of equilibria and bifurcations and illustrated this for both standard low dimensional as well as high dimensional SDEs. This does not replace existing approaches to analyse the dynamics of such stochastic systems either as random dynamical systems, through Lyapunov exponents, the approximation of manifolds, invariant measures or application of large deviation theory. Rather it offers a good starting point to investigate the effect of stochastic forcing on deterministic models of interest and to then apply these more sophisticated techniques for a refined study.
Since our approach is purely based on deterministic calculations it would be interesting to integrate this approach into existing numerical continuation software. This would allow, for example, two parameter continuation of bifurcations, tracking where $\MU=0$. This might be used to identify where stochastic forcing changes deterministic bifurcation points as a function of the two parameters.
The Lorenz example, Section~\ref{sec:lorenz}, illustrates that besides equilibria there are other interesting dynamical objects such as oscillations and limit cycles. Work is in progress to extend the results presented here for equilibria to periodic orbits. 
It would also be interesting to extend these results to other noise processes as well as to other underlying forms of dynamical systems.

\subsection*{Acknowledgments} The authors were funded by an International Centre for Mathematical Sciences Research in Groups award, for activity spanning 6--17 Jan 2025 that led to the inception of this work. For more information, see \url{https://icms.ac.uk/activities/research-in-group/oscillations-in-gene-regulatory-networks-stochastic-modelling-and-simulations/}.

\appendix
\section{Mean-square dissipativity and notions of stochastic stability.}
\label{sec:stabilityreview}

\subsection*{Notions of dissipativity}
The analysis in this article relied on the concept of mean-square dissipativity provided in Definition~\ref{Def:ms-dissipative}. The following, more general, definition is given by Schurz~\cite{Schurz2000}, reproduced here for $p=2$.
\begin{definition}
    \label{Def:ms-dissipative-schurz}
    Let $\left\{X(t):t\in[0,\infty)] \subseteq D \subseteq \mathbb{R}^d \right\}$ be a solution of the SDE \eqref{eq:main_d} corresponding to initial data $X(0)\in D$ which is measurable with respect to $\mathcal{F}_0$. The SDE is called 
    \begin{enumerate}
    \item mean-square dissipative if  there exist $r,R>0$ such that     $$\limsup_{t\to\infty}\expect{\|X(t)\|^2}<R^2$$
    whenever $\expect{\|X(0)\|^2}<r^2$;

    \item uniformly mean-square dissipative if  there exist $r,R>0$ such that $$\lim_{t\to\infty}\expect{\sup_{0\leq s \leq t}\|X(s)\|^2} <R^2$$
    whenever $\expect{\|X(0)\|^2} < r^2$.
    \end{enumerate}
\end{definition}
Schurz in \cite{Schurz2000} then considers a more general form of the dissipativity condition \eqref{eq:dissipative} 
allowing for time dependence in $\alpha_2$ and $\alpha_3$ (where we assumed constant $\alpha_2$ and $\alpha_3$).

To examine dissipativity in deterministic systems it is often assumed that $\langle x, F(x)\rangle \leq -\alpha_2 \|x\|^2$ holds for $\|x\|^2>R^2$. In the stochastic setting this suggests the following alternative to Assumption~\ref{ass:dissipative}.
\begin{assumption}
    \label{ass:dissipative_old}
    Let $R>0$ be fixed and $p\geq 2$ be given. Then for $\|x\|^2>R^2$ there is an $\alpha_2>0$ such that
\begin{equation}\label{eq:dissipative_old}
\langle x, F(x)\rangle + \frac{p-1}{2}\|G(x)\|^2_{\mathbf{F}}  \leq -\alpha_2\|x\|^2.   
\end{equation}
\end{assumption}
Similar conditions have been considered in \cite{Veretennikov,Khasminskii2012} for the drift and these have been termed as the Veretennikov–Khasminskii condition, see \cite{Butkovsky2014}. A similar condition also appears in more recent work ~\cite[Assumption 3.1]{liu2025ergodicestimatesonestepnumerical}. 
We note, however, that Assumption~\ref{ass:dissipative_old} implies 
Assumption~\ref{ass:dissipative}. In fact,  Assumption~\ref{ass:dissipative_old} implies that, in the mean-square sense, solutions either stay in a bounded region, or are attracted by this region exponentially fast.

\subsection*{Notions of stability}
We reproduce from \cite{KellyBuckwar2010} an overview of different notions of stability  used in the literature for SDEs (see also \cite{Khasminskii2012}) that may be applied in the special case where $G(\Xc)=0$, i.e.~the deterministic equilibrium $\Xc$ is also an equilibrium of the SDE \eqref{eq:main_d}, and hence, $Y\equiv 0$  is an equilibrium of the centred SDE \eqref{eq:main_YSDE}.

\begin{definition}\label{def:classicalStochStab}
Suppose that $x^*$ is an equilibrium of the SDE \eqref{eq:main_d} and let
$X$ be a global solution with initial data $X_0\in\mathbb{R}^d$. Then $Y=X-x^*$ is a solution of the centered SDE \eqref{eq:main_YSDE} with $Y_0=X_0-x^*$. We say that the equilibrium  $x^*$ is
\begin{enumerate}
\item stable in probability for $t\geq 0$ if  for any  $s\geq 0$ and $\varepsilon>0$,
\[
\lim_{Y_0\to 0}\mathbb{P}\left[\sup_{t>s}\|Y(t)\|>\varepsilon\right]=0;
\]
\item asymptotically stable in probability if  it is stable in probability and 
\[
\lim_{Y_0\to 0}\mathbb{P}\left[\lim_{t\to\infty}\|Y(t)\|=0\right]=1;
\]
\item pth-moment stable if for any $\varepsilon>0$, there exists a $\delta> 0$ such that
\[
\expect{\|Y(t)\|^p}<\varepsilon,\quad t\geq 0,
\]
whenever $\|Y_0\|^p<\delta$; if $p=2$, the equilibrium is said to be mean-square stable;
\item globally pth-moment asymptotically stable if  it is pth-moment stable, and for all $Y_0\in\mathbb{R}^d$,
\[
\lim_{t\to\infty}\expect{\|Y(t)\|^p}=0;
\] 
if $p=2$, the equilibrium is said to be mean-square asymptotically stable;
\item a.s.~stable if  for any $\varepsilon>0$ there exists $\delta> 0$ such that
\[
\|Y(t)\|<\varepsilon,\quad t\geq 0,\quad a.s,
\] 
whenever $\|Y_0\|<\delta$;
\item globally a.s.~asymptotically stable if  for all $Y_0\in\mathbb{R}^d$, 
\[
\lim_{t\to \infty}\|Y(t)\|=0\quad a.s.
\]
\end{enumerate}
\end{definition}

We observe that by Chebychev's inequality mean-square (asymptotic) stability implies (asymptotic) stability in probability of an equilibrium, but the converse is not true.

Next, we present a theorem from \cite{Khasminskii2012} that relates linear and nonlinear stability in probability. To this end we consider the centered nonlinear SDE \eqref{eq:main_YSDE} assuming that $Y\equiv0$ is an equilibrium. The corresponding linearisation is  \eqref{eq:sdeLinMult} with $\Lambda\equiv\Gamma_i\equiv0$, $i=1,\dots,m$, i.e.~a constant-coefficient linear system of the form
\begin{equation}\label{eq:SDEgenLinKh}
dX(t)=AX(t)dt+\sum_{i=1}^{m}B_iX(t)dW_i(t),\quad t\geq 0;\quad X(0)=X_0,
\end{equation}
where $A,B_1,\ldots,B_d\in\mathbb{R}^{d\times d}$. The following proposition, adapted from Theorem 7.1 in Khasminskii~\cite{Khasminskii2012} to our setting, provides conditions ensuring that the stability in probability of the zero solution of \eqref{eq:main_YSDE} is implied by that of the zero solution of~\eqref{eq:SDEgenLinKh}. 

\begin{proposition}\label{prop:linnonlin}

Suppose the solution $X\equiv 0$ of the linear system \eqref{eq:SDEgenLinKh} is globally a.s.~asymptotically stable (or asymptotically stable in probability), and there exists $\varepsilon,\gamma>0$ sufficiently small that the drift and diffusion coefficients of \eqref{eq:main_YSDE} satisfy 
\begin{equation*}
\|f(\xi)-A\xi\|+\sum_{i=1}^{m}\|g_i(\xi)-B_i \xi\|<\gamma \|\xi\|,
\end{equation*}
when $\xi\in\mathbb{R}^d$ satisfies $\|\xi\|<\varepsilon$, and  $g_i$ is the $i^{th}$ column  of $g$. Then the solution $Y\equiv 0$ of the SDE \eqref{eq:main_YSDE} is asymptotically stable in probability.
\end{proposition}
If $x^*$ is an equilibrium of the SDE \eqref{eq:main_d} we can apply  Proposition \ref{prop:linnonlin} to the centred SDE~\eqref{eq:main_YSDE} and its  linearisation~\eqref{eq:sdeLinMult}. Then, 
\begin{align*}
    f(\xi)-A\xi=R_F(y,x^*),\qquad  g_i(\xi)-B_i \xi=R_{G_i}(y,x^*), \quad i=1,\dots, m.
\end{align*}
Hence, if the remainder terms are small and zero is globally a.s.~asymptotically stable, (or asymptotically stable in probability) for \eqref{eq:sdeLinMult}, then the equilibrium $x^*$ is asymptotically stable in probability.

\printbibliography

\end{document}